\documentclass[11pt,oneside,english]{amsart}
\usepackage[T1]{fontenc}
\usepackage[latin9]{inputenc}
\usepackage{geometry}
\usepackage{textcomp}
\usepackage{amsthm}
\usepackage{esint}

\makeatletter
\numberwithin{equation}{section}
\numberwithin{figure}{section}
\theoremstyle{plain}
\newtheorem{thm}{\protect\theoremname}[section]
  \theoremstyle{plain}
  \newtheorem{cor}[thm]{\protect\corollaryname}
  \theoremstyle{plain}
  \newtheorem{lem}[thm]{\protect\lemmaname}
  \theoremstyle{plain}
  \newtheorem{prop}[thm]{\protect\propositionname}
  \theoremstyle{remark}
  \newtheorem{rem}[thm]{\protect\remarkname}
  \theoremstyle{definition}
  \newtheorem{example}[thm]{\protect\examplename}

\def\makebbb#1{
    \expandafter\gdef\csname#1\endcsname{
        \ensuremath{\Bbb{#1}}}
}\makebbb{R}\makebbb{N}\makebbb{Z}\makebbb{C}\makebbb{H}\makebbb{E}\makebbb{H}\makebbb{P}\makebbb{B}\makebbb{K}\makebbb{E}

\usepackage{babel}

\usepackage{babel}

\makeatother

\usepackage{babel}

\makeatother

\usepackage{babel}

\makeatother

\usepackage{babel}
  \providecommand{\corollaryname}{Corollary}
  \providecommand{\examplename}{Example}
  \providecommand{\lemmaname}{Lemma}
  \providecommand{\propositionname}{Proposition}
  \providecommand{\remarkname}{Remark}
\providecommand{\theoremname}{Theorem}

\begin{document}

\title{A thermodynamical formalism for Monge-Ampère equations, Moser-Trudinger
inequalities and Kähler-Einstein metrics }

\author{Robert J. Berman}

\email{robertb@chalmers.se}

\curraddr{Mathematical Sciences - Chalmers University of Technology and University
of Gothenburg - SE-412 96 Gothenburg, Sweden }

\keywords{Monge-Ampere equation, Kähler-Einstein manifolds, Variational methods
(MSC 2010: 32Q20, 32W20, 35A15)}
\begin{abstract}
We develop a variational calculus for a certain free energy functional
on the space of all probability measures on a Kähler manifold $X$.
This functional can be seen as a generalization of Mabuchi's $K-$energy
functional and its twisted versions to more singular situations. Applications
to Monge-Ampère equations of mean field type, twisted Kähler-Einstein
metrics and Moser-Trudinger type inequalities on Kähler manifolds
are given. Tian's $\alpha-$ invariant is generalized to singular
measures, allowing in particular a proof of the existence of Kähler-Einstein
metrics with positive Ricci curvature that are singular along a given
anti-canonical divisor (as conjectured very recently by Donaldson).
As another application we partially confirm a well-known conjecture
in Kähler geometry showing that if the Calabi flow in the (anti-)
canonical class exists for all times then it converges to a Kähler-Einstein
metric, when a unique one exists. 

\tableofcontents{} 
\end{abstract}
\maketitle

\section{Introduction}

One of the motivations for the present paper comes from the probabilistic
approach to Kähler-Einstein metrics very recently introduced in \cite{b1}.
In op. cit. the relations to physics were emphasized (Euclidean gravity
and fermion-boson correspondences) and a heuristic argument was given
for the convergence of the statistical mechanics model in the thermodynamical
limit. One of the aims of the present paper, which can be seen as
the first part in a forthcoming series, is to develop the \emph{variational
calculus }needed for a rigorous investigation of the thermodynamical
limit referred to above. However the main results to be proved also
have an independent interest in Kähler-Einstein geometry (notably
to the convergence of the Calabi flow and a conjecture of Donaldson
concerning Kähler-Einstein metrics on Fano manifolds which are singular
along a divisor) and more generally in the context of complex Monge-Ampère
equations and Moser-Trudinger type inequalities.

This work can also be seen as a development of the variational approach
to complex Monge-Ampère equations recently introduced in \cite{bbgz}.
The main role will be played by a certain functional $F(\mu)$ on
the space of all probability measures on a Kähler manifold that in
the thermodynamical limit, referred to above, arises as the limiting
\emph{free energy functional.} We will be particularly interested
in the \emph{optimizers} of $F(\mu)$ (as explained in \cite{b1}
they determine the limiting equilibrium measures of the statistical
mechanical model).

Using Legendre transforms the functional $F(\mu)$ will be related
to the another functional $\mathcal{G}(u)$ on the space of all singular
Kähler potentials (i.e. $\omega-$psh functions), which played a leading
role in \cite{bbgz}. As is will turn out the free energy functional
$F(\mu)$ can, in the particular case when the Kähler class is proportional
to the canonical class, be identified with \emph{Mabuchi's K-energy
functional, }which plays a leading role in Kähler-Einstein geometry.
As for the functional $\mathcal{G}(u)$ it generalizes a functional
introduced by Ding \cite{din} in Kähler-Einstein geometry that we
following \cite{rub1} will refer to as the \emph{Ding functional.}

From the point of view of Kähler geometry the main conceptual contribution
of the present paper is to introduce a thermodynamical formalism for
Kähler-Einstein geometry, which in mathematical terms amounts to a
systematic use of convexity and Legendre transform arguments and to
- which is closely related - develop a variational calculus for the
Mabuchi functional which demands a minimum of regularity assumptions,
namely finite (pluricomplex) energy and finite entropy.

\subsection{General setup}

Let $(X,\omega)$ be an $n-$dimensional compact complex manifold
with Kähler form $\omega$ and fix a probability measure $\mu_{0}$
on $X$ and non-zero real parameter $\beta$ (which plays the role
of the inverse temperature in the statistical mechanical setup in
\cite{b1}). To the triple $(\omega,\mu_{0},\beta)$ we will associate
a Monge-Ampère equation, as well as two functionals. Before continuing
it should be emphasized that only the Kähler\emph{ class} $[\omega]\in H^{1,1}(X)$
defined by the fixed Kähler form $\omega$ will be relevant and one
may as well fix any other smooth and, possibly non-positive, representative
$\omega'\in[\omega].$ We let $d^{c}:=i(-\partial+\overline{\partial})/4\pi,$
so that $dd^{c}=\frac{i}{2\pi}\partial\overline{\partial}$ and denote
by $V$ the volume of $(X,\omega),$ i.e. $V=\int_{X}\omega^{n}/n!$
which by Stokes theorem is an invariant of the class $[\omega].$

\subsubsection*{The Monge-Ampère mean field equation }

This is the following equation 
\begin{equation}
\frac{(\omega+dd^{c}u)^{n}}{Vn!}=\frac{e^{\beta u}\mu_{0}}{\int_{X}e^{\beta u}\mu_{0}}\label{eq:me intro}
\end{equation}
 for an $\omega-$psh function $u$ on $X,$ i.e. 
\begin{equation}
\omega_{u}:=\omega+dd^{c}u\geq0\label{eq:omega u pos intro}
\end{equation}
 in the sense of currents. The integral in the equation has been inserted
to ensure invariance under the additive action of $\R$ (removing
gives an equivalent equation) and hence the equation descends to the
space of all positive currents in the class $[\omega].$

The equation above generalizes the mean field equations extensively
studied on a Riemann surface, i.e the case when $n=1$ (see the book
\cite{ta} and references therein). It should be interpreted in the
weak sense of pluripotential theory as recalled in section \ref{sec:Functionals-on-the}.
More precisely, we will assume that the fixed measure $\mu_{0}$ has
finite energy and we will look for finite energy solutions. These
energy notions will be recalled in section \ref{sec:Functionals-on-the}.
One of the main cases that we will be interested in is when $\mu_{0}$
is a volume form and then we will simply look for\emph{ smooth} solutions
of the equation \ref{eq:me intro} satisfying 
\[
\omega_{u}:=\omega+dd^{c}u>0,
\]
 which means that $u$ is a \emph{Kähler potential} for the Kähler
metric $\omega_{u}$ in the cohomology class $[\omega].$ A Interestingly,
the equation \ref{eq:me intro} also has a natural interpretations
for $\beta=0,$ as well as $\beta=\infty.$ Indeed, for $\beta=0$
this is clearly the \emph{inhomogeneous Monge-Ampère equation }and
for $\beta=\infty$ it may be interpreted as a \emph{free boundary
value problem} for the Monge-Ampère equation (see Theorem \ref{thm:free bound}).

\subsection{The (twisted) Kähler-Einstein setting }

The case of main interest in Kähler geometry arises when the class
$[\omega]$ is a non-zero multiple of the\emph{ canonical class},
i.e. the first Chern class of the canonical line bundle $K_{X}:=\Lambda^{n}(TX^{*}):$
\[
[\omega]=\beta c_{1}(K_{X})
\]
 (after scaling we may and will assume that $\beta=\pm1)$ and when
the fixed Kähler form $\omega$ and measure $\mu_{0}$ are related
by 
\[
\mu_{0}=e^{-h_{\omega}}\omega^{n}/V
\]
 for the Ricci potential $h_{\omega}$ of the fixed Kähler metric
$\omega.$ Then the equation \ref{eq:me intro} is equivalent to the
\emph{Kähler-Einstein equation} 
\[
\mbox{Ric}\omega=-\beta\omega
\]
 where $\mbox{Ric}\omega$ denotes the Ricci form defined by the Ricci
curvature of the Riemannian metric determined by $\omega.$ By the
seminal results of Aubin \cite{au} and Yau \cite{y} such a Kähler-Einstein
metric always exists in the case when $\beta\geq0.$ But it is well-known
that there are obstructions to the existence of Kähler-Einstein metrics
in the case when $\beta<0,$ i.e. when $X$ is a Fano manifold. The
Yau-Tian-Donaldson conjeceture (see \cite{do,ti} and references therein)
formulates these obstructions in terms of an algebro-geometric notion
of \emph{stability} (in the sense of Geometric Invariant Theory).
Even though there has been tremendous progress on this conjecture,
which was settled on complex surfaces by Tian \cite{t-y}, it is still
open in dimension $n\geq3.$ However, as shown by Tian (see \cite{ti})
there is a stronger \emph{analytic }notion of stability which is equivalent
to the existence of a Kähler-Einstein metric in the class $\beta c_{1}(K_{X}),$
namely the \emph{properness} of \emph{Mabuchi's K-energy functional
}$\mathcal{K}$ (which in this case turns out to be equivalent to
the\emph{ coercivity} of the functional \cite{p-s+}).\emph{ }The
functional $\mathcal{K}$ is defined on the space of all Kähler metrics
in $\beta c_{1}(K_{X})$ and its critical points are precisely the
Kähler-Einstein metrics. In the case of a general class $[\omega]$
and volume form $\mu_{0}$ the equation \ref{eq:me intro} is equivalent
to a\emph{ twisted} Kähler-Einstein equation (see section \ref{sub:The-(twisted)-K=00003D0000E4hler-Einstein})
obtained by replacing $\mbox{Ric}\omega$ with $\mbox{Ric}\omega-\theta$
for a given closed real $(1,1)-$form $\theta.$

\subsection{Monge-Ampère mean field equations and Moser-Trudinger type inequalities:
General results }

The free energy functional $F_{\beta}$ of a probability measure $\mu$
of finite (pluricomplex) energy is defined as 
\begin{equation}
F_{\beta}=E_{\omega}+\frac{1}{\beta}D_{\mu_{0}}\label{eq:def of f intro}
\end{equation}
where $E_{\omega}(\mu)$ is the (pluricomplex)\emph{ energy }of the
probability measure $\mu$ introduced in \cite{bbgz} and $D_{\mu_{0}}$
is its \emph{entropy} relative to $\mu,$ which in the regular case
means that $D_{\mu_{0}}(\mu):=\int_{X}\log(\frac{\mu}{\mu_{0}})\mu.$
We will start by relating properties of the free energy functional
$F_{\beta}$ to another funtional functional $\mathcal{G}_{\beta}(u)$
defined on the space of all $\omega-$psh functions with finite energy
$\mathcal{E}^{1}(X,\omega).$ We refer to section \ref{sec:Monge-Amp=00003D0000E8re-mean-field}
for precise definitions. For the moment we just point out that the
functionals $F_{\beta}$ and $\mathcal{G}_{\beta}$ have an independent
analytical interest when $\beta<0.$ For example, on a Riemann surface
their boundedness from above is equivalent to a \emph{logarithmic
Hardy-Sobolev inequality} and \emph{Moser-Trudinger inequality} ,
respectively (which in turn imply various limiting Sobolev inequalities)\emph{
}\cite{Bec,c-l}.

In the analytically most challenging case when $\beta<0$ the main
properties that will be obtained are summarized in the following Theorem
(see section \ref{sub:Properness-and-coercivity} for the definition
of properness and coercivity in this context). 
\begin{thm}
\label{thm:duality intro}For any given measure $\mu_{0}$ of finite
energy and number $\beta<0$ we have 
\begin{equation}
\sup_{\mu\in E_{1}(X)}F_{\beta}=\sup_{u\in\mathcal{E}^{1}(X,\omega)}\mathcal{G}_{\beta}\label{eq:same optimal values intro}
\end{equation}
 and 
\begin{equation}
F_{\beta}(\frac{(\omega+dd^{c}u)^{n}}{Vn!})\leq\mathcal{G}_{\beta}(u)\label{eq:f smaller than g intro}
\end{equation}
 for any $u\in\mathcal{E}^{1}(X,\omega)$ with equality iff $u$ is
a solution to the  equation \ref{eq:me intro}. Moreover, the functional
$F_{\beta}$ is coercive iff $\mathcal{G}_{\beta}$ is. 
\end{thm}
In the Kähler-Einstein setting and when $u$ is assumed to be a Kähler
potential - so that $\beta F_{\beta}$ may be identified with Mabuchi's
K-energy functional and $\beta G_{\beta}(u)$ is the Ding functional
- the content of the previous theorem was previously known. Indeed,
the equality \ref{eq:same optimal values intro} was established by
Li \cite{li}, who used the Kähler-Ricci flow and Perelman's deep
estimates and by Rubinstein \cite{rub0,rub1}, using the Ricci iteration.
As for the inequality \ref{eq:f smaller than g intro} it follows
from identities of Bando-Mabuchi \cite{b-m}, while the coercivity
statement only has a rather involved and indirect proof (see section
\ref{sub:Intermezzo:-properness-vs} for further discussion and references).
The present proof uses a simple Legendre duality argument and has
the virtue of being valid in the general singular setting.

Combining the properties \ref{eq:same optimal values intro} and \ref{eq:f smaller than g intro}
above with the variational approach introduced in \cite{begz} is
the key to the proof of the following general existence and convergence
result. 
\begin{thm}
\label{thm:main intro}Let $(X,\omega)$ be a compact Kähler manifold
and let $\mu_{0}$ be a probability measure on $X$ of finite energy. 
\begin{itemize}
\item When $\beta>0$ the   functional $F_{\beta}(\mu)$ admits a unique
\emph{minimizer }$\mu$ on the space $E_{1}(X,\omega)$ of finite
energy probability measures on $X.$ Its potential $u_{\mu}$ is the
unique solution (mod $\R)$ of the  equation \ref{eq:me intro}. 
\item When $\beta<0$ and the   functional $F_{\beta}(\mu)$ is assumed
bounded from above on the space $E_{1}(X,\omega)$ any maximizer $\mu$
(it it exists) has a potential $u_{\mu}$ solving the  equation \ref{eq:me intro}.
Moreover, under the stronger assumption that $F_{\beta-\delta}$ be
bounded from above for some $\delta>0$ (or equivalently, if $-F_{\beta}$
is coercive with respect to energy) a maximizer does exist. 
\end{itemize}
More generally, if the   functional $F_{\beta}$ is coercive on $E_{1}(X,\omega)$
with respect to energy, then any sequence $\mu_{j}$ in $E_{1}(X,\omega)$
such that $F(\mu_{j})$ converges to the minimal value of $\beta F$
converges (perhaps after passing to a subsequence if $\beta<0)$ to
an minimizer $\mu.$ In the case when $\mu_{0}=fdV$ for a volume
form $dV$ on $X$ and $f\in L^{p}(X,dV)$ for some $p>1$ the assumptions
about coercivity above may be replaced by properness.
\end{thm}
In the case when $\mu_{0}$ is a volume form the weak solutions of
the equation \ref{eq:me intro} produced above are automaticaly \emph{smooth}
as follows from \cite{ko,sz-to}. The existence of solutions to \ref{eq:me intro}
for $\beta=0$ was shown in \cite{g-z2}, building on \cite{y} (see
also \cite{begz}). As pointed out above, in the Kähler-Einstein setting
the existence result in the previous theorem was shown by Aubin and
Yau in the case when $\beta>0$ and by Tian in the case when $\beta<0.$
The usual existence proofs are based on the continuity method (compare
Remark \ref{rem:nadel seq}). However, in the general situation when
$\beta<0$ it does not seem possible (even when $\mu_{0}$ is a volume
form) to use a continuity method as there is no general uniqueness
result for the solutions (even modulo biholomorphisms), nor for the
solutions of the linearized equations and hence the crucial openness
property in the continuity method is missing in general.

To obtain natural situations where the coercivity assumption in the
previous theorem is satisfied ( for $\beta<0)$ we generalize Tian's
alpha-invariant of a Kähler class \cite{ti1} to an invariant $\alpha$
of a pair $([\omega],\mu_{0})$ leading to the following sufficent
criterion for coercivity of $F_{\beta}$ (see Theorem \ref{thm:alpha invariant intro}):
\begin{equation}
-\beta<\alpha(n+1)/n\label{eq:crit on alpha inv}
\end{equation}
This gives among other things a Moser-Trudinger type inequality for
Frostman measures on Riemann surfaces (see Cor \ref{cor:m-t on rieman sur intro}).

\subsection{Applications to the (twisted) Kähler-Einstein setting }

In the Kähler-Einstein setting the functional 
\begin{equation}
\mathcal{K}(u):=\beta F_{\beta}(\frac{\omega_{u}{}^{n}}{Vn!})\label{eq:mab f intro}
\end{equation}
 on the space $\mathcal{H}(X,\omega)$ of Kähler potentials for $[\omega]=\beta c_{1}(K_{X})$
will be shown to coincide with \emph{Mabuchi's K-energy functional}
\cite{m2} (see section \ref{sub:The-free-energy}). From Theorem
\ref{thm:main intro} we then deduce the first point in the following
Corollary (see Theorem \ref{cor:alpha inv for twisted} for the second
point): 
\begin{cor}
\label{cor:min of mab intro}Let $u$ be an $\omega-$psh function
with finite energy, i.e. $u\in\mathcal{E}^{1}(X,\omega).$ Then
\begin{itemize}
\item $u$ minimizes (with a finite minimum) the generalized Mabuchi functional
$\mathcal{K}$ iff $\omega_{u}$ is a Kähler-Einstein metric (and
in particular smooth and non-degenerate).
\item If $X$ is a Fano manifold with no non-trivial holomorphic vector
fields (i.e. $H^{0}(TX)=\{0\})$ and $u_{j}$ is a normalized minimizing
sequence for $\mathcal{K},$ i.e. $\sup_{X}u_{j}=0$ and 
\[
\mathcal{K}(u_{j})\rightarrow\inf_{\mathcal{H}(X,\omega)}\mathcal{K},
\]
 then precisely one of the following alternatives holds: (1) either
$X$ admits a Kähler-Einstein metric $\omega_{KE}$ and then $\omega_{u_{j}}$
converges weakly to $\omega_{KE}$ or (2) $u_{j}$ subconverges to
$u_{\infty}$ defining a Nadel multiplier ideal sheaf on $X,$ i.e.
$\int_{X}e^{-u_{\infty}t}dV=\infty$ for any $t>n/(n+1).$
\end{itemize}
\end{cor}
The first point above generalizes a recent result of Chen-Tian-Zhou
\cite{c-t-z}, saying that any maximizer $u$ such that $\omega_{u}$
has locally bounded coefficients is necessarily smooth and Kähler-Einstein.
It should be pointed out that the minimal assumption of finite energy
of the maximizer $u$ in the assumptions in Corollary \ref{cor:min of mab intro}
is crucial as there seems to be no known way of controlling the a
priori regularity of a general maximizer. In particular, this will
allow us to apply the previous corollary to the Calabi flow below.
As for the second point it can bee seen as a generalization of Nadel's
result concerning the continuity method \cite{na} (see Remark \ref{rem:nadel seq}).

\subsection{Application to The Calabi flow}

The Calabi flow\emph{ }\cite{ca} is the following flow of Kähler
metrics: 
\[
\frac{\partial\omega_{t}}{\partial t}=dd^{c}R_{\omega_{t}}
\]
 where $R_{\omega_{t}}$ is the scalar curvature of the Kähler metric
$\omega_{t},$ which is a highly non-linear $4$th order parabolic
PDE. It has been conjectured that the flow exists for all times and
it is expected to converge to a constant scalar curvature metric in
$[\omega]$ when such a metric exists \cite{ca,do}. In this direction
we will prove the following 
\begin{thm}
\label{thm:conv of cal intro}Let $[\omega]$ be a Kähler class such
that $[\omega]=\beta c_{1}(K_{X})$ for $\beta\neq0.$ In case $\beta<0$
we assume that $X$ admits a Kähler-Einstein metric $\omega_{KE}$
and that $H^{0}(TX)=\{0\}.$ If the Calabi flow $\omega_{t}$ exists
for all times $t\geq0,$ then it converges weakly to the Kähler-Einstein
metric, i.e. 
\[
\omega_{t}\rightarrow\omega_{KE},
\]
 as $t\rightarrow\infty$ holds in the weak topology of currents. 
\end{thm}
The existence and convergence of the Calabi flow on a Riemann surface
was shown by Chrusciel \cite{chr}. In the general higher dimensional
case almost all results are conditional. It was proved by Chen-He
\cite{c-h} that the Calabi flow exists as long as the Ricci curvature
stays uniformly bounded. Moreover, they obtained the convergence towards
an extremal metric (which in the case $[\omega]=\beta c_{1}(K_{X})$
is the Kähler-Einstein metric) under the extra assumption that the
potential $u_{t}$ be uniformly bounded along the flow The previous
theorem should be viewed in the light of the corresponding result
for the Kähler-Ricci flow in $\beta c_{1}(K_{X}).$ As shown by Cao
\cite{cao-1} this latter flow exists for all times, regardless of
the sign of $\beta,$ and converges to the Kähler-Einstein metric
when $\beta<0.$ However, the convergence towards a Kähler-Einstein
metric (when it exists) was only proved recently by Tian-Zhou \cite{t-z}
using the deep estimates of Perelman. The previous theorems extends
to the setting of twisted Kähler-Einstein metrics as long as the twisting
form $\theta$ is non-negative (see Remark \ref{rem:twisted calabi flow}).

\subsection{Applications to Donaldson's equation}

In section \ref{sub:K=00003D0000E4hler-Einstein-metrics-on log} we
will consider twisted Kähler-Einstein metrics for the singular twisting
form befined by the current of integration along a divisor on $X.$
We will be particularly interested in the case when $X$ is a Fano
manifold and the divisor $D$ is smooth and represents $c_{1}(-K_{X})$
and consider the following equation

\begin{equation}
\mbox{Ric}\omega_{\gamma}=\gamma\omega_{\gamma}+(1-\gamma)\delta_{D}\label{eq:donaldsons equ}
\end{equation}
 where $\gamma>$ and $\delta_{D}$ denotes the current of integration
along $D.$ The equation was recently studied by Donaldson who conjectured
that it admits solutions for $\gamma$ sufficently small. This is
confirmed by the following theorem formulated in terms of the ordinary
alpha-invariants of $-K_{X}$ and its restriction to $D:$ 
\begin{thm}
\label{thm:don eq}Let $X$ be a Fano manifold with a smooth anti-canonical
divisor $D.$ Let $\gamma$ be a fixed parameter such that 
\[
0<\gamma<\Gamma:=\frac{n+1}{n}\min\left\{ \alpha(-K_{X}),\alpha((-K_{X})_{|D})\right\} ,
\]
 (where $\Gamma>0).$ 
\begin{itemize}
\item There is a smooth Kähler-Einstein metric $\omega_{\gamma}$ on $X-D$
such that $\omega_{\gamma}$ has Hölder continuous local potentials
on all of $X$ and such that equation \ref{eq:donaldsons equ} holds
globally on $X.$ Moreover, $\omega_{\gamma}\geq\omega$ for some
Kähler form $\omega$ on $X.$ 
\item the metric $\omega_{\gamma}$ is unique (among all metrics with bounded
potentials) and $\gamma\mapsto\omega_{\gamma}$ $(\gamma\in]0,\Gamma[)$
is a continuous curve in the space of Kähler currents on $X$ and
the restriction to $X-D$ gives a continuous curve in the space of
all Kähler forms on $X-D$ equipped with the $\mathcal{C}^{\infty}-$topology
on compacts. 
\end{itemize}
\end{thm}
Donaldson proposed a program for producing Kähler-Einstein metric
by first obtaining solutions to equation \ref{eq:donaldsons equ}
for some $\gamma=\gamma_{0}$ and then deforming $\gamma$ up to $\gamma=1$
(using an assumption of K-stability in the last step). More precisely,
in Step 1 in the notes \cite{do-2}, p.33, it was conjectured that
there is a solution for $\gamma_{0}$ sufficiently small, which moreover
has\emph{ }cone singularities along $D.$ It should be pointed out
that the Kähler-Einstein metric $\omega_{\gamma}$ on $X-D$ produced
in the proof of Theorem \ref{thm:don eq}, a priori, only has a volume
form with cone singularities along $D.$ However, in the orbifold
case, i.e. when $\gamma=1-1/m$ for some positive integer $m,$ it
follows from standard arguments \cite{t-y0} that the metric $\omega_{\gamma}$
itself has cone singularities (see the discussion in section \ref{sec:Donaldson's-equation-and}
for the general case). Donaldson proposed solving the equation \ref{eq:donaldsons equ}
for $\gamma$ sufficiently small by perturbing the complete Ricci
flat metric of Tian-Yau on $X-D$ which, at least formally, is a solution
of equation \ref{eq:donaldsons equ} when $\gamma=0$ \cite{t-y-1}.
This can be seen as a singular variant of the usual continuity method. 

One virtue of the present approach is that it bypasses the openness
problem in the proposed continuity method. The key point of the proof
is to study how the alpha-invariant of the pair $(X,(1-\gamma)D)$
depends on the parameter $\gamma.$ This will allow us to show that
the twisted Mabuchi K-energy $\mathcal{K}_{(1-\gamma)D}$ is coercive
when $\gamma<\Gamma.$ Then the previous variational approach can
be used to produce a weak solution to equation \ref{eq:donaldsons equ}.
As for the uniqueness it is deduced from Berndtsson's very recent
generalized Bando-Mabuchi theorem \cite{bern}, saying that uniqueness
holds for solutions to equations of the form \ref{eq:donaldsons equ},
given a smooth divisor $D,$ unless there are non-trivial holomorphic
vector fields on $X$ tangent to $D.$ In our case the non-existence
of such vector fields follows from the properness of $\mathcal{K}_{(1-\gamma)D},$
which, as explained above, holds for any positive $\gamma$ which
is sufficiently small.

It should be pointed out that in case of negative Ricci curvature
the existence of Kähler-Einstein metric with conical singularities
along a divisor was previously conjectured by Tian \cite{ti3} in
connection to applications to algebraic geometry and further studied
by Jeffres \cite{j} and Mazzeo \cite{ma} (where an existence result
was announced for $\gamma\in]0,1/2]$). See the end of section \ref{sec:Donaldson's-equation-and}
for a futher discussion of very recent developments concerning cone
singularities.

\subsection*{Organization}

In Section \ref{sec:Functionals-on-the} we setup the pluripotential
theoretic and functional analytical framework, emphasizing the role
of Legendre transforms (in infinite dimensions). In section \ref{sec:Monge-Amp=00003D0000E8re-mean-field}
the main results concerning general Monge-Ampère mean field equations
stated in the introduction are proved. In the following sections these
results are applied and refined in the setting of twisted Kähler-Einstein
metrics (section \ref{sub:The-(twisted)-K=00003D0000E4hler-Einstein}),
the Calabi flow (section \ref{sec:Convergence-of-the-cal}) and log
Fano manifolds and Donaldson's equation (section \ref{sub:K=00003D0000E4hler-Einstein-metrics-on log}).
In the appendix we generalize some results of Demailly on the relation
between alpha-invariants and log canonical thresholds to the setting
of klt pairs.

\subsection*{Acknowledgments}

I am very grateful to Sébastien Boucksom, Vincent Guedj and Ahmed
Zeriahi for the stimulating collaboration \cite{bbgz} which paved
the way for the present paper. Also thanks to Bo Berndtsson for discussions
related to \cite{bern}, to Yanir Rubinstein for once sending me his
thesis where I learned about the $C^{2}-$estimate in \cite{b-k,rub1}
and for giving many useful comments on drafts of the present paper
and also thanks to Valentino Tosatti and Gabor Székelyhidi for their
comments.

\subsection*{Notational remark}

Throughout, $C,$ $C'$ etc denote constants whose values may change
from line to line

\section{\label{sec:Functionals-on-the}Functionals on the spaces of probability
measures and $\omega-$psh functions and Legendre duality}

In this section we will consider various functionals defined on the
space $\mathcal{M}_{1}(X)$ of probability measures on $X$, as well
as on the space $PSH(X,\omega)$ of \emph{$\omega-$psh functions}
on $X$ (also called \emph{potentials}). It will be important to also
work with different subspaces of these spaces:

\[
\begin{array}{rccccc}
V(X):= & \mbox{\ensuremath{\left\{ Volume\,\ forms\right\} }} & \subset & E_{1}(X,\omega) & \subset & \mathcal{M}_{1}(X)\\
\mathcal{H}(X,\omega):= & \mbox{\ensuremath{\left\{ K\ddot{a}hler\,\ potentials\right\} }} & \subset & \mathcal{E}^{1}(X,\omega) & \subset & PSH(X,\omega)
\end{array}
\]
 where $E_{1}(X,\omega)$ and $\mathcal{E}^{1}(X,\omega)$ are the
subspaces of \emph{finite energy} elements. These notions are higher
dimensional versions of the energy notions familiar from the classical
theory of Dirichlet spaces on Riemann surfaces. The general definitions
and relations to Legendre transforms will be recalled below.

\subsection{Functional analytic framework and Legendre-Fenchel transforms}

We equip the space $\mathcal{M}(X)$ of all signed finite Borel measures
on $X$ with its  usual weak topology, i.e. $\mu_{j}\rightarrow\mu$
iff 
\[
\left\langle u,\mu_{j}\right\rangle :=\int_{X}u\mu_{j}\rightarrow\int_{X}u\mu
\]
 for any continuous function $u,$ i.e. for all $u\in C^{0}(X).$
In other words, $\mathcal{M}(X)$ is the topological dual of the vector
space $C^{0}(X).$ We will be mainly concerned with the subspace $\mathcal{M}_{1}(X)$
of all probability measures on $X$ which is a convex compact subset
of $\mathcal{M}(X).$ This latter space is a locally convex topological
vector space. As such it admits a good duality theory (see section
4.5.2 in \cite{de-ze}): given a functional $\Lambda$ on the vector
space $C^{0}(X)$ its \emph{Legendre(-Fenchel) transform} is the following
functional $\Lambda^{*}$ on $\mathcal{M}(X):$ 
\[
\Lambda^{*}(\mu):=\sup_{u\in C^{0}(X)}(\Lambda(u)-\left\langle u,\mu\right\rangle )
\]
 Conversely, if $H$ is a functional on the vector space $\mathcal{M}(X)$
we let 
\[
H^{*}(u):=\inf_{\mu\in\mathcal{M}(X)}(H(\mu)+\left\langle u,\mu\right\rangle )
\]
 Note that we are using rather non-standard sign conventions. In particular,
$\Lambda^{*}(\mu)$ is always \emph{convex }and \emph{lower semi-continuous
(lsc),} while $H^{*}(u)$ is \emph{concave} and \emph{upper-semicontinuos
(usc). }As a well-known consequence of the Hahn-Banach separation
theorem we have the following fundamental duality relation (Lemma
4.5.8 in \cite{de-ze}): 
\begin{equation}
\Lambda=(\Lambda^{*})^{*}\label{eq:dualiy}
\end{equation}
 iff $\Lambda$ is concave and usc. We also recall the following basic
fact (we will not use the uniqueness property, only the minimization
property) 
\begin{lem}
\label{lem:maxim is differ}Assume that $\Lambda$ is a functional
on $C^{0}(X)$ which is finite, concave and Gateaux differentiable
(i.e differentiable along lines). Then, for a fixed $u\in C^{0}(X)$
the differential $d\Lambda_{|u}$ is the unique minimizer of the following
functional on $\mathcal{M}(X):$ 
\begin{equation}
\mu\mapsto\Lambda^{*}(\mu)+\left\langle u,\mu\right\rangle \label{eq:functional in lemma max diff}
\end{equation}
 (and the minimum value equals $\Lambda(u)).$\end{lem}
\begin{proof}
As a courtesy to the reader we give the simple proof. By the duality
relation \ref{eq:dualiy} the minimal value of the functional \ref{eq:functional in lemma max diff}
is indeed $\Lambda(u),$ which means that $\mu_{u}$ is a minimizer
iff 
\[
\Lambda(u)\leq\Lambda(u')+\left\langle u-u',\mu_{u}\right\rangle 
\]
 for all $u'\in C^{0}(X).$ When $\mu=d\Lambda_{|u}$ the previous
inequality follows immediately from the concavity of $\Lambda.$ More
generally, any $\mu_{u}$ satisfying the previous inequality is called
a \emph{subdifferential} for $\Lambda$ at $u.$ To prove uniqueness
we take $u'=u+tv$ for $v\in C^{0}(X)$ and $t\in\R$ and divide the
previous inequality by $t,$ letting $t$ tend to zero, first for
$t>0$ and then for $t<0,$ giving 
\[
\frac{d\Lambda(u+tv)}{dt}_{t=0^{-}}\leq\left\langle v,\mu_{u}\right\rangle \leq\frac{d\Lambda(u+tv)}{dt}_{t=0^{+-}}
\]
 Since $\Lambda$ is Gateaux differentiable the left and right derivative
above coincide forcing $\left\langle v,\mu_{u}\right\rangle =\left\langle v,d\Lambda_{|u}\right\rangle $
for any $v\in C^{0}(X).$ 
\end{proof}
Conversely, if the functional in the statement of the lemma above
has a unique maximizer $\mu_{u}$ then $\Lambda$ is Gateaux is differentiable
with $d\Lambda_{|u}=\mu_{u}.$ We will prove a variant of this fact
in Prop \ref{pro:deriv of E} below.

\subsection{The space $PSH(X,\omega)$ of $\omega-$psh\label{sub:The-space-}
functions}

A general reference for this section is \cite{g-z}. The space $PSH(X,\omega)$
of \emph{$\omega-$psh functions} (sometimes simple called \emph{potentials})
is defined as the space of all functions $u\in L^{1}(X)(:=L^{1}(X,\omega^{n})$
with values in $[-\infty,\infty[$ which are upper semi-continuous
and such that 
\[
\omega_{u}:=\omega+dd^{c}u\geq0
\]
 in the sense of currents. We endow $PSH(X,\omega)$ with the $L^{1}-$topology.
There is a basic continuous bijection \cite{g-z} 
\[
u\mapsto\omega_{u},\,\,\, PSH(X,\omega)/\R\leftrightarrow\left\{ \mbox{positive closed currents in \ensuremath{[\omega]}}\right\} 
\]
 where the right hand side is equipped with the weak topology (and
the space coincides with $\mathcal{M}_{1}(X)$ when $n=1$ and $V=1).$
In particular, this shows that $PSH(X,\omega)/\R$ is \emph{compact.
}The subspace of all \emph{Kähler potentials} is defined by 
\[
\mathcal{H}(X,\omega):=\left\{ u\in\mathcal{C^{\infty}}(X):\,\omega_{u}>0\right\} 
\]
 so that $\mathcal{H}(X,\omega)/\R$ is isomorphic to the space of
all Kähler forms in the class $\ensuremath{[\omega].}$ By the fundamental
approximations results of Demailly \cite{dem0} $\mathcal{H}(X,\omega)$
is dense in $PSH(X,\omega).$ See also \cite{bl-k} for a simple proof
of the last statement in the following proposition. 
\begin{prop}
\label{pro:(Demailly-ref).-The} The space $\mathcal{H}(X,\omega)$
is dense in $PSH(X,\omega)$ (wrt the $L^{1}-$topology): 
\[
PSH(X,\omega)=\overline{\mathcal{H}(X,\omega)}
\]
 More precisely, any $\omega-$psh function can be written as a decreasing
limit of elements $u_{j}$ in $\mathcal{H}(X,\omega).$ 
\end{prop}

\subsection{\label{sub:The-Monge-Amp=00003D0000E8re-operator}The Monge-Ampère
operator and the functional $\mathcal{E_{\omega}}(u)$}

In this section and the following one we recall notions and results
from \cite{g-z2,begz,bbgz} (a part from Prop \ref{pro:deriv of E},
which is new). Let us start by recalling the definition of the Monge-Ampère
measure $MA(u)$ on\emph{ smooth} functions. It is defined by 
\[
MA(u):=\frac{(\omega+dd^{c}u)^{n}}{Vn!}=:\frac{(\omega_{u})^{n}}{Vn!}
\]
 which is hence a (positive) probability measure when $u\in PSH(X,\omega).$
The Monge-Ampère $MA$ operator may be naturally identified with a
one-form on the vector space $C^{\infty}(X)$ by letting 
\[
\left\langle MA_{|u},v\right\rangle :=\int_{X}MA(u)v
\]
 for $u\in C^{\infty}(X).$ As observed by Mabuchi \cite{m2,m1} (in
the context of Kähler-Einstein geometry) the one-form $MA$ is closed
and hence it has a primitive $\mathcal{E}_{\omega}$ (defined up to
an additive constant) on the space all smooth weights, i.e. 
\begin{equation}
d\mathcal{E}_{|u}=MA(u)\label{eq:def prop of energy on psh}
\end{equation}
 We fix the additive constant by requiring $\mathcal{E}_{\omega}(0)=0.$
Integrating $\mathcal{E}_{\omega}$ along line segments one arrives
at the following well-known formula:
\begin{equation}
\mathcal{E}_{\omega}(u):=\frac{1}{(n+1)!V}\sum_{j=0}^{n}\int_{X}u\omega_{u}^{j}\wedge(\omega)^{n-j}.\label{eq:bi-energy}
\end{equation}
Conversely, one can simply take this latter formula as the definition
of $\mathcal{E}_{\omega}$ and observe that the following proposition
holds (compare \cite{b-b} for a more general singular setting): 
\begin{prop}
\label{pro:of energy}The following holds

$(i)$ The differential of the functional $\mathcal{E}_{\omega}$
at a smooth function $u$ is represented by the measure $MA(u),$
i.e. 
\begin{equation}
\frac{d}{dt}_{t=0}(\mathcal{E}_{\omega}(u+tv))=\int_{X}MA(u)v\label{eq:diff}
\end{equation}

$(ii)$ $\mathcal{E}_{\omega}$ is increasing on the space of all
smooth $\omega-$psh functions

$(iii)$ $\mathcal{E}_{\omega}$ is concave on the space of all smooth
smooth $\omega-$psh functions and when $n=1$ it is concave on all
of $C^{\infty}(X)$ 
\end{prop}
Note that $(ii)$ is a direct consequence of $(i),$ since the differential
of $\mathcal{E}_{\omega}$ is represented by a (positive) measure.

Following \cite{bbgz} we will sometimes refer to $\mathcal{E}_{\omega}$
as the \emph{Aubin-Mabuchi functional} (not to be confused with Mabuchi's
K-energy functional).

\subsubsection{The general singular setting}

One first extends the Aubin-Mabuchi functional $\mathcal{E}_{\omega}$
(formula \ref{eq:bi-energy}) to all $\omega-$psh functions by defining
\[
\mathcal{E}_{\omega}(u):=\inf_{u'\geq v}\mathcal{E}_{\omega}(u')\in[-\infty,\infty[
\]
 where $u$ ranges over all locally bounded (or smooth) $\omega-$psh
functions $u'$ such that $u'\geq u.$ Next, we let 
\[
\mathcal{E}^{1}(X,\omega):=\{u\in PSH(X,\omega):\,\mathcal{E}_{\omega}(u)>-\infty\},
\]
that we will refer to as the space of all $\omega-$psh functions
with\emph{ finite (pluri-)energy. }In the Riemann surface case $\mathcal{E}^{1}(X,\omega)$
is the classical Dirichlet subspace of $PSH(X,\omega)$ consisting
of all functions whose gradient is in $L^{2}(X).$

As a consequence of the monotonicity of $\mathcal{E}_{\omega}(u)$
and Bedford-Taylor's fundamental local continuity result for mixed
Monge-Ampère operators one obtains the following proposition (cf.
\cite{begz}, Prop 2.10; note that $\mathcal{E}_{\omega}=-E_{\chi}$
for $\chi(t)=t$ in the notation in op. cit.) 
\begin{prop}
\label{pro:energy is ups}The functional $\mathcal{E}_{\omega}(u)$
is upper semi-continuous on $PSH(X,\omega),$ concave and non-decreasing.
Moreover, it is continuous wrt decreasing sequences in $PSH(X,\omega).$
\end{prop}
For any $u\in\mathcal{E}^{1}(X,\omega)$ the (non-pluripolar) Monge-Ampère
measure $MA(u)$ is well-defined \cite{begz} and does not charge
any pluripolar sets. We collect the continuity properties that we
will use in the following \cite{begz} 
\begin{prop}
\label{pro:non-pluripol ma}Let $u_{i}$ be a sequence decreasing
to $u\in\mathcal{E}^{1}(X,\omega).$ Then, as $i\rightarrow\infty,$
\[
MA(u_{i})\rightarrow MA(u)
\]
and 
\[
u_{i}MA(u_{i})\rightarrow uMA(u)
\]
in the weak topology of measures and $\mathcal{E}_{\omega}(u_{j})\rightarrow\mathcal{E}_{\omega}(u).$ 
\end{prop}
In particular, by the previous proposition we could as well have defined
$MA(u)$ for $u\in\mathcal{E}^{1}(X,\omega)$ as the limit of the
volume forms $MA(u_{j})$ with $u_{j}\in\mathcal{H}(X,\omega)$ any
sequence decreasing to $u$ (using Prop \ref{pro:(Demailly-ref).-The}).

\subsection{\label{sub:The-pluricomplex-energy}The pluricomplex energy $E(\mu)$
and potentials of measures}

Following \cite{begz} we define the (pluricomplex) energy by 
\begin{equation}
E(\mu):=\sup_{u\in PSH(X,\omega)}\mathcal{E}_{\omega}(u)-\left\langle u,\mu\right\rangle \label{eq:def of e as sup}
\end{equation}
 if $\mu\in\mathcal{M}_{1}(X).$ It will also be useful to extend
$E$ to all of the vector space $\mathcal{M}(X)$ by letting $E(\mu)=\infty$
on $\mathcal{M}(X)-\mathcal{M}_{1}(X).$ We will denote the subspace
of all finite energy probability measures by 
\[
E_{1}(X,\omega):=\{\mu:\, E(\mu)<\infty\}
\]

By Propositions \ref{pro:non-pluripol ma} and \ref{pro:(Demailly-ref).-The}
it is enough to take the sup over the subspace $C^{0}(X)\cap PSH(X,\omega)$
or even over the space $\mathcal{H}(X,\omega)$ of Kähler potentials.
But one point of working with less regular functions is that the sup
can be attained. Indeed, as recalled in the following theorem 
\begin{equation}
E(\mu):=\mathcal{E}_{\omega}(u_{\mu})-\left\langle u_{\mu},\mu\right\rangle \label{eq:energy in terms of pot}
\end{equation}
 for a unique function $u_{\mu}\in\mathcal{E}^{1}(X,\omega)/\R$ of
$\mu$ if $E(\mu)<\infty$ where 
\begin{equation}
MA(u_{\mu})=\mu.\label{eq:potential of meas}
\end{equation}
We will refer to a solution $u_{\mu}$ of the previous equation is
a \emph{potential} of $\mu$ (this is a somewhat non-standard terminology
as potentials usually are associated with closed $(1,1)-$currents,
rather then measures). 
\begin{thm}
\label{thm:var sol of ma}\cite{bbgz}The following is equivalent
for a probability measure $\mu$ on $X:$ 
\begin{itemize}
\item $E(\mu)<\infty$ 
\item \textup{$\left\langle u,\mu\right\rangle <\infty$ for all $u\in\mathcal{E}^{1}(X,\omega)$} 
\item $\mu$ has a potential $u_{\mu}\in\mathcal{E}(X,\omega$), i.e. equation
\ref{eq:potential of meas} holds 
\end{itemize}
Moreover, $u_{\mu}$ is a maximizer of the functional $\mathcal{E}_{\omega}-\left\langle \cdot,\mu\right\rangle $
and if $u_{j}$ is any sequence in $\mathcal{E}^{1}(X,\omega)$ such
that $\sup_{X}u_{j}=0$ and 
\[
\liminf_{j}\mathcal{E}_{\omega}(u_{j})-\left\langle u_{j},\mu\right\rangle \geq E(\mu)
\]
 then $u_{j}\rightarrow u_{\mu}$ where $u_{\mu}$ is the unique potential
of $\mu$ such that $\sup_{X}u_{\mu}=0$
\end{thm}
The previous theorem was proved in \cite{bbgz} using the variational
approach in the more general setting of a big class $[\omega].$ In
the case when $\mu$ is a volume form the Calabi-Yau theorem \cite{y}
furnishes a unique \emph{smooth} potential $u_{\mu}$ as above (using
the continuity method).

We will next prove a dual version of \ref{eq:def prop of energy on psh}
which is a new result in the general non-smooth setting. If the functional
$\left\langle \mu,\cdot\right\rangle $ were lsc on all of $\mathcal{E}^{1}(X,\omega)$
then the proposition below would essentially be a consequence of the
extremal property of $u_{\mu}$ given by Theorem \ref{thm:var sol of ma}
combined with a dual version of the converse of Lemma \ref{lem:maxim is differ}
on $\mathcal{M}(X).$ 
\begin{prop}
\label{pro:deriv of E}Let $\mu^{t}=\mu^{0}+t\nu$ be a segment in
$E_{1}(X,\omega):=\{E<\infty\}$ where $t\in]-\epsilon,\epsilon[$
for some $\epsilon>0.$ Then 
\begin{equation}
\frac{dE(\mu^{t})}{dt}_{|t=t_{0}}=-\int_{X}u_{\mu^{t_{0}}}\nu,\label{eq:der of E}
\end{equation}
 where $u_{\mu_{t}}$ is the potential of $\mu$ (which is unique
mod $\R).$ Moreover, for any two elements $\mu^{1}$ and $\mu^{0}$
of $E_{1}(X,\omega)$ we have 
\begin{equation}
E(\mu^{1})\geq E(\mu^{0})+\int_{X}(-u_{\mu^{0}})(\mu^{1}-\mu^{0}),\label{eq:conv ineq for E}
\end{equation}
 \end{prop}
\begin{proof}
Denote by $u^{t}$ the potential of $\mu^{t}$ normalized so that
$\sup u^{t}=0.$ Then 
\begin{equation}
\frac{1}{t}E\left((\mu^{t})-E(\mu^{0})\right)=\frac{1}{t}\left((\mathcal{E}_{\omega}(u^{t})-\left\langle u^{t},\mu^{0}\right\rangle )-(\mathcal{E}_{\omega}(u^{0})-\left\langle u^{0},\mu^{0}\right\rangle )\right)-\left\langle u^{t}-u^{0},\nu\right\rangle +\left\langle -u^{0},\nu\right\rangle \label{eq:beginning pf of deriv of E}
\end{equation}
 \emph{Step one:} $\left\langle u^{t}-u^{0},\nu\right\rangle \rightarrow0$
as $t\rightarrow0.$

First observe that there is a constant $C$ such that 
\[
\mbox{Claim\,1: }u^{t}\in\left\{ \mathcal{E}_{\omega}\geq-C\right\} \cap\{\sup_{X}=0\}
\]

Indeed, by the extremal property of $u^{t}$ we have $\mathcal{E}_{\omega}(u^{t})-(\left\langle u^{t},\mu^{0}\right\rangle +t\left\langle u^{t},\nu\right\rangle )=$
\[
\mathcal{E}_{\omega}(u^{t})-\left\langle u^{t},\mu^{t}\right\rangle \geq\mathcal{E}_{\omega}(u^{0})-\left\langle u^{0},\mu^{t}\right\rangle =C-t\left\langle u^{0},\nu\right\rangle \geq C''
\]
 Moreover, as shown in \cite{bbgz} (Prop 3.4), for any $\mu\in E_{1}(X,\omega)$
there is a constant $C_{\mu}$ such that 
\begin{equation}
|\left\langle u^{t},\mu\right\rangle |\leq C_{\mu}(-\mathcal{E}(u^{t}))^{1/2}\label{eq:bound on integr pairing}
\end{equation}
 if $u^{t}\in\left\{ \mathcal{E}_{\omega}>-\infty\right\} \cap\sup=0.$
Combining this latter inequality with the previous ones gives 
\[
\mathcal{E}_{\omega}(u^{t})\geq-C''-C'''(1+t)(-\mathcal{E}(u^{t}))^{1/2}
\]
 which proves the claim (since $t$ is bounded).

Next, we will prove the following 
\[
\mbox{Claim\,2:\,}\liminf_{t\rightarrow0}\mathcal{E}_{\omega}(u^{t})-\left\langle u^{t},\mu^{0}\right\rangle \geq\mathcal{E}_{\omega}(u^{0})-\left\langle u^{0},\mu^{0}\right\rangle 
\]
 As above, by the extremal property of $u^{t}$ it is enough to prove
that 
\[
\left\langle u^{t},\mu^{t}\right\rangle -\left\langle u^{t},\mu^{0}\right\rangle =t\left\langle u^{t},\nu\right\rangle \rightarrow0
\]
 as $t\rightarrow0.$ But this follows from the upper bound \ref{eq:bound on integr pairing}
combined with claim 1 above.

Now, Claim 2 combined with the last statement in Theorem \ref{thm:var sol of ma}
shows that $u^{t}\rightarrow u^{0}$ in $L^{1}(X,\omega^{n})$ when
$t\rightarrow0.$ As shown in \cite{bbgz} for any $\mu\in E_{1}(X,\omega)$
(and trivially also for the difference $\nu$ of elements in $E_{1}(X,\omega))$
the functional $\left\langle \cdot,\mu\right\rangle $ is continuous
wrt the $L^{1}-$topology on the subset in the Claim 1. This finishes
the proof of step one.

\emph{Step two: proof of formula \ref{eq:der of E}}

By concavity the function of $t$ inside the first bracket in the
rhs of \ref{eq:beginning pf of deriv of E} achieves its maximum on
$]-\epsilon,\epsilon[$ at the value $t=0$ and hence letting $t\rightarrow0^{+}$
gives 
\[
\frac{dE(\mu^{t})}{dt}_{t=0^{+}}\leq0+0-\left\langle -u^{0},\nu\right\rangle 
\]
 Similarly, 
\[
\frac{dE(\mu^{t})}{dt}_{t=0^{-}}\geq0+0-\left\langle -u^{0},\nu\right\rangle 
\]
 But by the convexity of $E(\mu^{t})$ we have $\frac{dE(\mu^{t})}{dt}_{t=0^{-}}\leq\frac{dE(\mu^{t})}{dt}_{t=0^{+}}$
which finally proves the equality \emph{\ref{eq:der of E}.}

\emph{Step three: proof of inequality \ref{eq:conv ineq for E}}

Let now $\mu_{t}$ be the affine segment, with $t\geq0,$ connecting
the given points $\mu^{0}$ and $\mu^{1}$ Combining the convexity
of $E(\mu^{t})$ and formula \emph{\ref{eq:der of E}} (evaluated
at $t=t_{0}>0)$ we have 
\[
E(\mu^{1})\geq E(\mu^{t_{0}})+\int_{X}(-u_{\mu^{t_{0}}})(\mu^{1}-\mu^{0})(1-t_{0})
\]
 and hence letting $t_{0}\rightarrow0$ and using step one above and
the fact that $E$ is lower semi-continuous finishes the proof of
the proposition. 
\end{proof}
Note that since the integral of $\nu$ vanishes the derivative above
is independent of the normalization of $u_{\mu}$.

Before continuing we note that $E(\mu)$ is \emph{not} (at least as
it stands) a Legendre transform of $\mathcal{E}_{\omega}(u)$ even
when restricted to $\mathcal{M}_{1}(X),$ because as explained above
the sup must be taken over the convex \emph{subspac}e $C^{0}(X)\cap PSH(X,\omega)$
of the vector space $C^{0}(X)$ In order to realize $E$ as a Legendre
transform we turn to the definition of the projection operator $P_{\omega}.$ 
\begin{rem}
\label{rem:legendre when n is one}When $n=1$ the sup referred to
above may actually by taking over all of $C^{0}(X).$ Indeed, as explained
above the extremizer $u_{\mu}$ a posteriori satisfies $\omega_{u_{\mu}}=\mu\geq0$
and hence $E$ is indeed the Legendre transform of $\mathcal{E}^{*}$
in the Riemann surface case. 
\end{rem}

\subsection{\label{sub:The-psh-projection}The psh projection $P$ and the formula
$E=(\mathcal{E}\circ P)^{*}$ }

Consider the following projection operator $P_{\omega}:\, C^{0}(X)\rightarrow C^{0}(X)\cap PSH(X,\omega)$
\[
P_{\omega}u:=\sup\left\{ v(x):\,\, v\in PSH(X,\omega),\,\,\, v\leq u\,\mbox{on\,\ensuremath{X}}\right\} 
\]
 (the lower semi-continuity of $P_{\omega}u$ follows from \ref{pro:(Demailly-ref).-The}
which allows us to write $P_{\omega}u$ as an upper envelope of continuous
functions and the upper semi-continuity is obtained by noting that
$P_{\omega}u$ is a candidate for the sup in its definition). One
of the main results in \cite{b-b} is the following 
\begin{thm}
(B.-Boucksom \cite{b-b}) The functional $\mathcal{E_{\omega}}\circ P_{\omega}$
is concave and Gateaux differentiable on $C^{0}(X).$ More precisely,
\[
d(\mathcal{E_{\omega}}\circ P_{\omega})_{|u}=MA(\mathcal{E_{\omega}}(P_{\omega}u))
\]

\end{thm}
The differentiability of the composed map $\mathcal{E_{\omega}}\circ P_{\omega}$
should be contrasted with the fact that the non-linear projection
$P_{\omega}$ is certainly not differentiable. The main ingredient
in the proof of the previous theorem is the following orthogonality
relation: 
\begin{equation}
\left\langle MA(Pu),(u-Pu)\right\rangle =0,\label{eq:orthog relation}
\end{equation}
 Note that it follows immediately from the fact that $Pu\leq u$ that
\begin{equation}
E=(\mathcal{E_{\omega}}\circ P_{\omega})^{*}\,\,\mbox{on\,\ensuremath{\mathcal{M}_{1}(X)}}\label{eq:e as leg on prob}
\end{equation}
 Moreover, by the previous theorem 
\begin{equation}
d(\mathcal{E_{\omega}}\circ P_{\omega})(C^{0}(X))\subset\mathcal{M}_{1}(X)\subset\mathcal{M}(X)\label{eq:diff is prob}
\end{equation}
In particular we obtain the flowing proposition (which is a slight
refinement of Theorem 5.3 in \cite{bbgz}): 
\begin{prop}
\label{pro:legendre duality for energies}The relation \ref{eq:e as leg on prob}
holds on all of the vector space $\mathcal{M}(X)$ of signed measures
on $X,$ i.e. 
\[
E=(\mathcal{E_{\omega}}\circ P_{\omega})^{*}\,\,\mbox{on\,\ensuremath{\mathcal{M}(X)}}
\]
 and dually 
\[
\mathcal{E_{\omega}}\circ P_{\omega}=E^{*}\,\,\mbox{on\,\ensuremath{\mathcal{C}^{0}(X)}}
\]
 \end{prop}
\begin{proof}
Since by definition $E=\infty$ on $\mathcal{M}(X)-\mathcal{M}_{1}(X)$
we have for any $u\in C^{0}(X)$ 
\[
E^{*}(u):=\inf_{\mu\in\mathcal{M}(X)}(E(\mu)+\left\langle u,\mu\right\rangle )=\inf_{\mu\in\mathcal{M}_{1}(X)}(E(\mu)+\left\langle u,\mu\right\rangle ),
\]
 and hence the identity \ref{eq:e as leg on prob} combined with \ref{eq:diff is prob}
and Lemma \ref{lem:maxim is differ} (not using the uniqueness) gives,
with $\Lambda:=\mathcal{E}\circ P,$ 
\[
E^{*}(u):=\inf_{\mu\in\mathcal{M}_{1}(X)}(\Lambda^{*}(\mu)+\left\langle u,\mu\right\rangle )=\inf_{\mu\in\mathcal{M}(X)}(\Lambda^{*}(\mu)+\left\langle u,\mu\right\rangle )
\]
 Finally, by the duality relation \ref{eq:dualiy} this means that
$E^{*}(u)=(\Lambda^{*})$ and applying the Legendre transform again
also gives $E=\Lambda^{*}.$ 
\end{proof}
In particular, if follows immediately from the previous proposition
that 
\[
E^{*}=\mathcal{E_{\omega}}\,\,\mbox{on\,\ensuremath{\mathcal{C}^{0}(X)\cap PSH(X,\omega)}}
\]

\subsection{\emph{\label{sub:The-relative-entropy}}The relative entropy $D(\mu)$
and its Legendre transform $\mathcal{L}^{-}$}

The relative entropy $D_{\mu_{0}}(\mu):=D(\mu)$ wrt a fixed probability
measure $\mu_{0}$ is defined by 
\[
D(\mu):=\int_{X}\log(\mu/\mu_{0})\mu
\]
 when $\mu$ is absolutely continuous wrt $\mu_{0}$ and otherwise
$D(\mu):=\infty.$ As is well-known $D$ is the Legendre transform,
i.e. $D=\mathcal{L}^{*},$ of the following functional on $C^{0}(X):$
\[
\mathcal{L}_{\mu_{0}}(u):=-\log\int_{X}e^{u}\mu_{0}
\]
 (compare the proof of Lemma \ref{lem:sup of l attained}). More generally,
for any given parameter $\beta\in\R-\{0\}$ and measurable function
$u,$ 
\[
\mathcal{L}_{\mu_{0,\beta}}(u):=-\frac{1}{\beta}\log\int_{X}e^{\beta u}\mu_{0}
\]
 which in particular defines a functional on $C^{0}(X)$ which, by
Hölder's inequality is concave for $\beta>0$ and convex for $\beta<0.$
The following basic duality relation holds when $\beta>0$ (Lemma
6.2.13 in \cite{de-ze}): 
\[
\mathcal{L}_{\beta}^{*}(-\mu)=\frac{1}{\beta}D(\mu)
\]
 i.e. 
\[
\frac{1}{\beta}D(\mu)=\sup_{u\in C^{0}(X)}\left(-\frac{1}{\beta}\log\int_{X}e^{\beta u}\mu_{0}+\left\langle u,\mu\right\rangle \right)
\]
 Similarly, if $\beta=-\gamma$ with $\gamma>0$ then we have that
\[
\mathcal{L}_{\gamma}^{-}(u):=-\mathcal{L}_{\mu_{0,-\gamma}}(u):=-\frac{1}{\gamma}\log\int_{X}e^{-\gamma u}\mu_{0}
\]
 is a concave functional and by symmetry 
\[
\mathcal{L}_{\gamma}^{-}{}^{*}=\frac{1}{\gamma}D
\]
 i.e. 
\[
\frac{1}{\gamma}D(\mu)=\sup_{u\in C^{0}(X)}\left(-\frac{1}{\gamma}\log\int_{X}e^{-\gamma u}\mu_{0}-\left\langle u,\mu\right\rangle \right)
\]
 Note that on $C^{0}(X)$ it follows directly from the chain rule
that 
\[
d\mathcal{L}_{\gamma}^{-}=\frac{e^{-\gamma u}\mu_{0}}{\int_{X}e^{-\gamma u}\mu_{0}}
\]
 so that the image of $C^{0}(X)$ under $d\mathcal{L}_{\gamma}^{-}$
is the subspace of $\mathcal{M}_{1}(X)$ of all measures $\mu$ with
strictly positive continuous density wrt $\mu_{0}.$ However we will
need to calculate the derivatives with almost no regularity assumptions. 
\begin{prop}
\label{pro:deriv of D}Let $\mu^{t}=\mu^{0}+t\nu$ be a segment in
$\{D<\infty\}.$ Then
\[
\frac{dD(\mu^{t})}{dt}_{t=0^{+}}=\int_{X}\log(\mu^{0}/\mu_{0})\nu
\]
 if the right hand side above is finite. Similarly, let $u^{t}=u+tv$
be a segment in the space of all usc functions where $\mathcal{L}_{\gamma}^{-}(u)$
is finite. Then 
\[
\frac{d\mathcal{L}_{\gamma}^{-}(u^{t})}{dt}_{t=0^{+}}=\int_{X}\frac{ve^{-\gamma u}\mu_{0}}{\int_{X}e^{-\gamma u}\mu_{0}}
\]

\end{prop}
if the right hand side above is finite. 
\begin{proof}
By definition 
\[
\frac{1}{t}\left(D(\mu^{t})-D(\mu^{0})\right)=\int_{X}\frac{1}{t}(\log(\frac{\mu^{t}}{\mu_{0}})-\log(\frac{\mu^{0}}{\mu_{0}}))\mu^{0}+\int_{X}\log(\frac{\mu^{t}}{\mu_{0}})\nu
\]
 Since $x\mapsto\log x$ is monotone and convex with derivative $1/x$
when $x>0$ the integrands above are monotone in $t$ and hence the
monotone convergence theorem gives 
\[
\frac{dD(\mu^{t})}{dt}_{t=0^{+}}=\int_{X}\frac{\nu}{\mu^{0}}\mu^{0}+\int_{X}\log(\frac{\mu^{0}}{\mu_{0}})\nu
\]
 By assumption $\int_{X}\nu=0$ and hence the first term above vanishes
which proves the first formula in the proposition.

The second formula of the theorem is proved in a similar fashion now
using that $x\mapsto e^{x}$ is convex (exactly as in the proof of
Lemma 6.1 in \cite{bbgz}) 
\end{proof}
Now we can prove the following 
\begin{lem}
\label{lem:sup of l attained}Let $\mu$ be a finite energy measure
and assume that $u\in\mathcal{E}^{1}(X,\omega)$ with $\int_{X}e^{-\gamma u}\mu_{0}<\infty.$
Then
\begin{equation}
(\mathcal{L}_{\gamma}^{-})^{*}(\mu)=\log(-\frac{1}{\gamma}\int_{X}e^{-\gamma u}\mu_{0})-\left\langle u,\mu\right\rangle (:=\mathcal{N}(u))\label{eq:max for l}
\end{equation}
 iff 
\end{lem}
\begin{equation}
\mu=\frac{e^{-\gamma u}\mu_{0}}{\int_{X}e^{-\gamma u}\mu_{0}}\label{eq:mu as exp}
\end{equation}

\begin{proof}
First note that by the assumptions on $u$ and $\mu$ both terms in
the definition of $\mathcal{N}(u)$ above are finite. Assume first
that $u$ satisfies \ref{eq:mu as exp}. If $v$ denotes a fixed continuous
function on $X$ and $u_{t}:=u+tv,$ then according to the previous
proposition 
\begin{equation}
\frac{d(\mathcal{N}(u_{t}))}{dt}_{t=0^{+}}=0\label{eq:vanish of der of n}
\end{equation}
 By concavity it follows that $\mathcal{N}(u)\geq\mathcal{N}(u+tv)$
for any $t\geq0$ and in particular for $t=1.$ Now take an arbitrary
function $w\in C^{0}(X)$ and write the lsc function $w-u$ as an
increasing limit of continuous functions $v_{j}.$ Since, as explained
above, 
\[
\mathcal{N}(u)\geq\mathcal{N}(u+v_{j})
\]
 letting $j\rightarrow\infty$ and using the monotone convergence
theorem gives 
\[
\mathcal{N}(u)\geq\sup_{w\in C_{0}(X)}\mathcal{N}(w):=(\mathcal{L}_{\gamma}^{-})^{*}(\mu)
\]
 Similarly, writing $u$ as a decreasing limit of continuous functions
$w_{j}$ and passing to the limit forces equality above.

Conversely, assume that $u$ satisfies \ref{eq:max for l} above.
Then it follows in particular (approximating as above) that the differentiable
function 
\[
t\mapsto\mathcal{N}(u_{t})
\]
 with $u_{t}$ as above attains its maximum at $t=0.$ Hence, the
critical point equation \ref{eq:vanish of der of n} holds and since
$v$ was arbitrary it follows by the formula in the previous proposition
that $u$ satisfies the relation \ref{eq:mu as exp}. 
\end{proof}

\subsection{\label{sub:Properness-and-coercivity}Properness and coercivity of
functionals}

The energy functional $E$ defines an exhaustion function on the space
$E_{1}(X,\omega)$ (i.e. the sets $\{E\geq-C\}$ are compact, since
$E$ is lsc, and their union is $E_{1}(X,\omega)).$ A functional
$F(\mu)$ on $E_{1}(X,\omega)$ is said to be \emph{proper (wrt energy)
}if it is proper with respect to the previous exhaustion, i.e. 
\[
E(\mu)\rightarrow\infty\implies F(\mu)\rightarrow\infty
\]
 and\emph{ coercive} (which is a stronger condition) if it there are
positive constants $a$ and $b$ such that 
\[
F\geq aE-b
\]
 Similarly, the functional $-\mathcal{E}_{\omega}$ defines an exhaustion
function on the space $\mathcal{E}^{1}(X,\omega)$ (it is indeed lsc
according to \ref{pro:energy is ups}). To get an exhaustion function
of $\mathcal{E}^{1}(X,\omega)/\R$ one replaces $-\mathcal{E}_{\omega}$
with its $\R-$invariant analogue 
\[
J_{\omega}(u):=-\mathcal{E}_{\omega}(u)+\int_{X}u\frac{\omega^{n}}{Vn!}
\]
 often called Aubin's $J-$functional in the Kähler geometry literature.
This then gives a notion of properness (wrt energy) and coercivity
on $\mathcal{E}^{1}(X,\omega)/\R,$ as well, introduced by Tian in
the setting of Kähler geometry (see \cite{ti} and references therein)

In fact, the notions of properness and coercivity above are preserved
under the bijection 
\[
\mathcal{E}^{1}(X,\omega)/\R\rightarrow E_{1}(X,\omega):\,\,\,\, u\mapsto MA(u)
\]
 as follows from the following basic lemma, which also involves Aubin's
\emph{$I-$functional:} 
\[
I_{\omega}(u):=-\frac{1}{Vn!}\int u(\omega_{u}^{n}-\omega^{n})
\]

\begin{lem}
\label{lem:energy as i minus j}The following identity holds 
\[
E(MA(u))=(I_{\omega}-J_{\omega})(u)
\]
 and 
\begin{equation}
\frac{1}{n}J_{\omega}\leq(I_{\omega}-J_{\omega})\leq nJ_{\omega}\label{eq:ineq for i and j}
\end{equation}
In particular, if $\mu\in E_{1}(X,\omega)$ with potential $u_{\mu}\in\mathcal{E}^{1}(X,\omega),$
normalized so that $\int u_{\mu}\omega^{n}=0,$ then 
\begin{equation}
-\left\langle u_{\mu},\mu\right\rangle \geq(\frac{n+1}{n})E(\mu)\label{eq:energy bounded by pairing}
\end{equation}

\end{lem}

\section{\label{sec:Monge-Amp=00003D0000E8re-mean-field}Monge-Ampère mean
field equations and Moser-Trudinger type inequalities }

Fix a probability measure $\mu_{0}$ of finite energy. Recall that
$\beta$ denotes a fixed parameter in $\R-\{0\}$ and when $\beta<0$
we will often write $\beta=-\gamma.$

The \emph{(normalized) Monge-Ampère mean field equation} (ME) associated
to the triple $(\omega,\mu_{0},\beta)$ is the following equation
for $u\in\mathcal{E}^{1}(X,\omega)$ 
\begin{equation}
\frac{\omega_{u}^{n}}{Vn!}=\frac{e^{\beta u}\mu_{0}}{\int_{X}e^{\beta u}\mu_{0}}\label{eq:normalized me}
\end{equation}
 where we recall that the measure in the left hand side above is the
Monge-Ampère measure $MA(u).$ Thanks to the normalizing integral
the equation is invariant under the additive action of $\R$ on $\mathcal{E}^{1}(X,\omega).$
The \emph{non-normalized }ME is the equation

\begin{equation}
\frac{\omega_{u}^{n}}{Vn!}=e^{\beta u}\mu_{0}\label{eq:non-normalized}
\end{equation}
 whose solutions are precisely the solutions of \ref{eq:normalized me}
with $\int_{X}e^{\beta u}\mu_{0}=1.$ In general, the transformation
$u\mapsto u-\frac{1}{\beta}\log\int_{X}e^{\beta u}\mu_{0}$ clearly
maps solutions of \ref{eq:normalized me} to solutions of \ref{eq:non-normalized}.

In this section we will be mainly concerned with the corresponding
\emph{free energy functional} 
\[
F_{\beta}(\mu):=E_{\omega}(\mu)+\frac{1}{\beta}\int_{X}\log(\frac{\mu}{\mu_{0}})\mu
\]
 defined on the space $E_{1}(X,\omega)$ of measure $\mu$ of finite
(pluricomplex) energy (section \ref{sub:The-pluricomplex-energy}).
We recall that the integral in the second term (i.e. the relative
entropy) is by definition equal to $\infty$ if $\mu$ is not absolutely
continuous wrt $\mu_{0}.$ In particular, $F_{\beta}(\mu)$ takes
values in $]-\infty,\infty]$ when $\beta>0$ and in $[-\infty,\infty[$
when $\beta<0.$

One of the reasons that we assume that $\mu_{0}$ is of finite energy
is that we will be interested in the cases when $\beta<0$ and the
functional $F_{\beta}$ admits a maximizer and in particular when
it is bounded from above. But as pointed out below a necessary condition
for this is that $\mu_{0}$ be of finite energy (see the discussion
after Theorem \ref{thm:f and g when beta neg}). We will also be interested
in the closely related functional

\[
\mathcal{G}_{\beta}(u):=\mathcal{E}_{\omega}(u)-\frac{1}{\beta}\log\int_{X}e^{\beta u}\mu_{0}\in[-\infty,\infty[
\]
 defined on the space $\mathcal{E}^{1}(X,\omega)$ of finite energy
$\omega-$psh functions (see section \ref{sec:Monge-Amp=00003D0000E8re-mean-field}).
To avoid notational complexity we will sometimes omit the subscripts
$\beta,$ (as well as the explicit dependence on $\omega$ and $\mu_{0})$. 

We start with the following general regularity result whose first
part is obtained by combining \cite{ko} and \cite{sz-to}. 
\begin{prop}
\label{pro:reg}If $\mu_{0}$ is a volume form then any solution $u\in\mathcal{E}^{1}(X,\omega)$
to equation \ref{eq:non-normalized} is smooth. More generally, the
solution is Hölder continuous under any of the following assumptions: 
\begin{itemize}
\item \cite{ko2}~~$\mu_{0}=fdV$ where $f\in L^{p}(X,dV)$ for some $p>1$
and where $dV=\omega_{0}^{n}$ is the volume form on $X$ of the metric
$\omega_{0}.$ 
\item \cite{hi}~~$\beta\geq0$ and $\mu=fdV_{M}$ where $f\in L^{p}(X,\mu)$
where $M$ is a real smooth submanifold $M$ of $X$ which has codimension
one and $dV_{M}$ is the measure supported on $M$ obtained by integrating
against the Riemannian volume form on $M$ induced by $\omega_{0}$ 
\end{itemize}
\end{prop}
\begin{proof}
\emph{Let $\mu_{0}$ be a volume form and }$u\in\mathcal{E}^{1}(X,\omega)$
a solution to equation \ref{eq:non-normalized}.\emph{ Step one: $u$
is bounded (continuous). }Since $u\in\mathcal{E}^{1}(X,\omega)$ the
function $u$ has no Lelong numbers (\cite{g-z2}, Cor 1.8) , i.e.
$\int e^{\beta u}\mu_{0}$ is integrable for all $\beta$ (by Skoda's
inequality, see for example \cite{d-j}). In particular, by equation
\ref{eq:non-normalized} $MA(u)\in L^{p}(X)$ for some $p>1.$ But
then Kolodziej's theorem \cite{ko} says that $u$ is bounded (and
even Hölder continuous \cite{ko2}).

\emph{Step two: higher order regularity. }By the previous step $u$
is a bounded weak solution to an equation of the form $MA(u)=e^{\Phi(u)}\mu_{0}$
where $\Phi(x)$ is a smooth function on $\R.$ But then the theorem
of Székelyhidi-Tosatti \cite{sz-to} says that $u$ is smooth. 
\end{proof}
When $\beta=0$ the first and second point is proved in \cite{ko2}
and \cite{hi}, respectively. But then the case when $\beta>0$ also
follows, since the factor $f:=e^{\beta u}$ is always bounded then
(just using that $u$ is usc).

\subsection{The case when $\beta>0$}

We start by considering the general case when $\beta>0$ which is
considerably simpler than the case when $\beta<0.$ This difference
in behavior is a reflection of the fact that in the former case the
functional $F_{\beta}$ above is a \emph{sum} of two convex functionals,
while in the latter case it is a \emph{difference} of two convex functionals.
The following theorem gives a slightly more general version of Theorem
\ref{thm:main intro} stated in the introduction, in the case $\beta>0,$
as we do not assume that $\mu_{0}$ has finite energy. 
\begin{thm}
\label{thm:existence etc when beta pos}Assume that $\beta>0$ and
that the background measure $\mu_{0}$ does not charge pluripolar
sets. Then there is a unique solution $u_{ME}\in\mathcal{E}_{1}(X)$
mod $\R$ of the  equation \ref{eq:normalized me}. Moreover, $u_{ME}$
is smooth if $\mu_{0}$ is a volume form. In general, 
\begin{itemize}
\item $u_{ME}$ is the unique (mod $\R)$ maximizer of the functional $\mathcal{G_{\beta}}$
on $\mathcal{E}^{1}(X,\omega)$ 
\item $\mu_{ME}(:=MA(u_{ME}))$ is the unique minimizer of the free energy
functional $F_{\beta}$ on $\mathcal{M}_{1}(X)$ 
\end{itemize}
\end{thm}
More generally, if $\mu_{j}$ is a sequence such that 
\[
F_{\beta}(\mu_{j})\rightarrow\inf_{E(\mu,\omega)}F_{\beta}
\]
 then $\mu_{j}$ converges to $\mu_{ME}$ in the weak topology of
measures. 
\begin{proof}
To simplify the notation we assume that $\beta=1$ and write $\mathcal{G}:=\mathcal{G}_{1}$
and $\mathcal{L}^{+}(u):=\log\int e^{u}\mu_{0}$ so that $\mathcal{G}=\mathcal{E}-\mathcal{L}^{+}.$

\emph{Existence of solution:}

The existence of a solution $u_{ME}\in\mathcal{E}^{1}(X,\omega)$
is proved by adapting the variational approach to solving Monge-Ampère
equations introduced in \cite{bbgz} to the present setting. In the
paper \cite{bbgz} the case when $\beta=0$ was treated, as well as
the case when $\beta>0$ and $\mu_{0}$ is a volume form.

\emph{Step one:} \emph{existence of a maximizer of $\mathcal{G}$}

We will denote by $\mathcal{E}^{1}(X,\omega)_{0}$ the subspace of
all $u$ in $\mathcal{E}^{1}(X,\omega).$ such that $\sup_{X}u=0.$
Since $\mathcal{G}$ is invariant under the $\R-$action we may take
a sequence in $\mathcal{E}^{1}(X,\omega)_{0}$ such that 
\[
\mathcal{G}(u_{j})\rightarrow\sup_{\mathcal{E}^{1}(X,\omega)}\mathcal{G}<\infty
\]
 Moreover, by the compactness of $PSH(X)/\R$ (see section \ref{sub:The-space-})
we may assume that $u_{j}\rightarrow u_{\infty}$ in $L^{1}(X).$
By Prop \ref{pro:energy is ups} $\mathcal{E}_{\omega}$ is usc and
according to Lemma 1.14 in \cite{b-b-w} so is $\mathcal{L}^{+}$
since $\mu_{0}$ does not charge pluripolar sets (see Lemma \ref{lem:kont of exp integral with bd}
below for a generalization). Hence $u_{\infty}\in\mathcal{E}^{1}(X,\omega)$
and 
\[
\mathcal{G}(u_{\infty})\geq\sup_{\mathcal{E}^{1}(X,\omega)}\mathcal{G}<\infty
\]
 and since $u_{\infty}$ is a candidate for the sup equality must
hold above.

\emph{Step two: Any maximizer of $\mathcal{G}$ on $\mathcal{E}^{1}(X,\omega)$
satisfies equation }\ref{eq:normalized me}

Let $u_{*}$ be a maximizer, fix $v\in\mathcal{C}^{\infty}(X)$ and
consider the following function on $\R$ 
\[
g(t):=\mathcal{E}(P(u_{*}+tv))+\mathcal{L}^{+}(u_{\infty}+tv)
\]
 where $(P(u_{*}+tv)\in\mathcal{E}^{1}(X,\omega),$ since $v$ is
bounded. It has a global maximizer at $t=0.$ Indeed, this using that
the projection $P$ and $-\mathcal{L}$ are increasing with respect
to $\leq$ gives 
\[
(\mathcal{E\circ}P)-\mathcal{L})(u)=(\mathcal{E}-\mathcal{L})(Pu)+\mathcal{L\circ}P-\mathcal{L\le}(\mathcal{E}-\mathcal{L})(Pu)
\]

Since by Theorem \ref{thm:var sol of ma} (and a simple approximation
argument; see Lemma 4.2 in \cite{bbgz}) and Prop \ref{pro:deriv of D}
$g$ is differentiable it follows from the formulas for their derivatives
that 
\[
\frac{dg(t)}{dt}_{t=0}=0=\left\langle MA(Pu_{*})-e^{\beta u}/\int_{X}e^{\beta u}\mu_{0},v\right\rangle =0
\]
 and since, by definition, $Pu_{*}=u_{*}$ and $v$ was arbitrary
this means that $u_{\infty}$ solves equation \ref{eq:normalized me}.

\emph{Regularity:}

By Prop \ref{pro:reg} any weak solution as above is in fact smooth
when $\mu_{0}$ is a volume form. It should be pointed out that when
$\mu_{0}$ is a volume form the existence of a smooth solution, when
$\beta>0,$ is a direct consequence of the Aubin-Yau estimates \cite{au,y},
using the continuity method.

\emph{Proof of the second point: $MA(u_{ME})$ is the unique minimizer
of $F$ (and $u_{ME}$ is the unique solution of equation \ref{eq:non-normalized})}

To prove this first observe that $F(\mu)$ is strictly convex on $\{F<\infty\}.$
Indeed, $E(\mu)$ is clearly convex (as it can be realized as a Legendre-Fenchel
transform) and it is well-known \cite{de-ze} that $D(\mu)$ is strictly
convex on $\{D<\infty\}.$ Now fix $\mu$ such that $F(\mu)<\infty$
and consider the affine segment 
\[
\mu^{t}:=\mu_{MF}(1-t)+t\mu=:\mu_{MF}+t\nu
\]
 Next let us prove that 
\[
\frac{dD(\mu^{t})}{dt}_{t=0^{+}}=-\int_{X}\log(\mu_{MF}/\mu_{0})\nu
\]
 But this follows from Prop \ref{pro:deriv of D}, since the rhs above
is finite. Indeed, by the  equation, it equals $-\int_{X}u_{MF}\nu$
where, as shown above, $\mathcal{E}(u_{MF})>\infty$ and by assumption
$\nu$ is a difference of finite energy measures. Hence, the integral
is finite according to Theorem \ref{thm:var sol of ma}. Now combining
the formula for $\frac{dD(\mu^{t})}{dt}_{t=0^{+}}$ above with the
convexity of $D$ and the inequality for $E$ in Prop \ref{pro:deriv of E}
gives 
\[
F(\mu)\geq F(\mu_{ME})+0
\]
 for any $\mu$ such that $E(\mu)$ and $D(\mu)$ are both finite.
Moreover, the \emph{strict} concavity of $F$ discussed above shows
that $\mu_{ME}:=MA(u_{ME})$ is the \emph{unique} minimizer of $F(\mu)$
on $E_{1}(X,\omega).$ The previous argument also gives that any solution
$u_{1}\in\mathcal{E}^{1}(X,\omega)$ of equation \ref{eq:non-normalized}
is such that $MA(u_{1})$ is a minimizer of $F.$ As a consequence
$MA(u_{0})=MA(u_{1})$ for any two solutions and hence $u_{0}-u_{1}$
is constant according to Theorem \ref{thm:var sol of ma}. This finishes
the proof of the second point.

To prove the final convergence recall that the functionals $E$ and
$D$ arise as Legendre transforms and are in particular lower semi-continuous.
As a consequence any weak limit point $\mu_{*}$ of the sequence $\mu_{j}$
is a minimizer of $F(\mu).$ But then it follows from the strict convexity
used above (i.e. the uniqueness) that $\mu_{*}=\mu_{ME}.$ 
\end{proof}
Now we can prove the following theorem \ref{thm:main intro} (which
in particular implies Theorem \ref{thm:duality intro} in the case
$\beta>0):$ 
\begin{thm}
\label{thm:relat f and g beta pos}Assume that $\beta>0.$ Then the
following relations between the functionals $F:=F_{\beta}$ and $\mathcal{G}:\mathcal{=}\mathcal{G}_{\beta}$
hold \end{thm}
\begin{itemize}
\item For any $u\in\mathcal{E}^{1}(X,\omega)$ we have 
\[
F(MA(u))\geq\mathcal{G}(u)
\]
 and 
\[
F(e^{\beta u}\mu_{0}/\int e^{\beta u}\mu_{0})\geq\mathcal{G}(u)
\]
 Equality in any of the two inequalities above holds iff $u$ is a
solution of the  equation \ref{eq:normalized me} (and hence equalities
then hold in \emph{both} inequalities above) 
\item Moreover, 
\[
\inf_{\mu\in E_{1}(X,\omega)}F(\mu)=\sup_{u\in\mathcal{E}^{1}(X,\omega)}\mathcal{G}(u)<\infty
\]
 \end{itemize}
\begin{proof}
We skip the proof of the first point as it is a trivial modification
(obtained by changing a few signs) of the proof given below for the
corresponding inequalities in Theorem \ref{thm:f and g when beta neg}.
The first point then immediately gives 
\begin{equation}
\inf_{\mu\in E_{1}(X,\omega)}F(\mu)\geq\sup_{u\in\mathcal{E}^{1}(X,\omega)}\mathcal{G}(u)\label{eq:proof of f and g beta pos}
\end{equation}
 According to the previous theorem the infimum in the LHS above is
attained precisely for $\mu=MA(u)$ where $u$ is the unique solution
mod $\R$ of the  equation \ref{eq:normalized me} and similarly for
the supremum in the RHS above. But then it follows from the equality
case in the first point that equality in fact holds in \ref{eq:proof of f and g beta pos}. 
\end{proof}

\subsection{The case when $\beta<0$}

In this case we start by proving the following refinement of Theorem
\ref{thm:duality intro}, in the case when $\beta<0:$
\begin{thm}
\label{thm:f and g when beta neg}The following relations between
the functionals $F:=F_{-\gamma}$ and $\mathcal{G}:\mathcal{=}\mathcal{G}_{-\gamma}$
hold \end{thm}
\begin{itemize}
\item The suprema coincide
\begin{equation}
\sup_{\mu\in E_{1}(X,\omega)}F(\mu)=\sup_{u\in\mathcal{E}^{1}(X,\omega)}\mathcal{G}(u)\label{eq:same sup for beta neg}
\end{equation}
 
\item The following inequalities hold for any $u\in\mathcal{E}^{1}(X,\omega)$
\begin{equation}
F(MA(u))\leq\mathcal{G}(u)\label{eq:f smal then g}
\end{equation}
 and 
\begin{equation}
F(e^{-\gamma u}\mu_{0}/\int e^{-\gamma u}\mu_{0})\geq\mathcal{G}(u)\label{eq:lower bound on F by G}
\end{equation}
 Equality in any of the two inequalities above holds iff $u$ is a
solution of the  equation \ref{eq:normalized me} with $\beta=-\gamma$
(and hence equalities then hold in \emph{both} inequalities above). \end{itemize}
\begin{proof}
\emph{(of the first point):}

First recall the Legendre transform relations $E(\mu)=(\mathcal{E\circ}P)^{*}$
and $\frac{1}{\gamma}D(\mu)=\mathcal{L_{\gamma}}^{-}(\mu)^{*}$(see
section \emph{\ref{sub:The-relative-entropy}}). Let us first prove
\begin{equation}
\sup_{\mu\in E_{1}(X,\omega)}F(\mu)\geq\sup_{u\in\mathcal{E}^{1}(X,\omega)}\mathcal{G}(u)\label{eq:sup geq sup}
\end{equation}
 For the sake of notational simplicity we assume that $\gamma=1$
and simply write 
\[
\mathcal{L}(u):=\mathcal{L}_{-\gamma}^{-}(u):=-\log\int_{X}e^{-u}\mu_{0}
\]
 defining a concave functional on $C^{0}(X).$ First note that it
follows immediately from the definition of the Legendre transforms
that, 
\[
(\mathcal{E\circ}P)-\mathcal{L}(u)\geq c\,(\mbox{on \ensuremath{C^{0}(X))}}\Rightarrow(\mathcal{E\circ}P)^{*}(\mu)-\mathcal{L}^{*}(\mu)\geq c
\]
 and hence 
\[
\sup_{\mu\in E_{1}(X,\omega)}F(\mu)\geq\sup_{u\in C_{0}(X)}((\mathcal{E\circ}P)-\mathcal{L})(u)
\]
 Next, observe that 
\[
\sup_{u\in C^{0}(X)}((\mathcal{E\circ}P)-\mathcal{L})(u)=\sup_{u\in\mathcal{H}(X,\omega)}(\mathcal{E}-\mathcal{L})(u)
\]
 where the sup in the rhs may also be taken over $\mathcal{E}^{1}(X,\omega)$
Indeed, first using that the projection $P$ and $\mathcal{L}$ are
increasing with respect to the usual order relation on functions we
have 
\[
(\mathcal{E\circ}P)-\mathcal{L})(u)=(\mathcal{E}-\mathcal{L})(Pu)+\mathcal{L\circ}P-\mathcal{L\le}(\mathcal{E}-\mathcal{L})(Pu)
\]
 Hence comparing the value at $u$ in the lhs below with the value
at $Pu$ in the rhs below gives 
\[
\sup_{u\in C^{0}(X)}((\mathcal{E\circ}P)-\mathcal{L})(u)=\sup_{u\in C^{0}(X)\cap\mathcal{E}^{1}(X,\omega)}(\mathcal{E}-\mathcal{L})(u)
\]
 Finally, by Prop \ref{pro:(Demailly-ref).-The} any $u\in\mathcal{E}(X,\omega$
can be written as a decreasing limit of elements in $\mathcal{H}(X,\omega).$
Hence, by the continuity of $\mathcal{E}$ under such limits and Lebesgue's
monotone convergence theorem the restriction to $C^{0}(X)$ in the
rhs above may be removed, finishing the proof of the claim \ref{eq:sup geq sup}.

The reversed inequality in \ref{eq:sup geq sup} is proved by interchanging
the roles of $E(=(\mathcal{E\circ}P)^{*})$ and $(\mathcal{E\circ}P)$
and the roles of $\mathcal{L}^{*}$ and $\mathcal{L}$ and using the
duality relations in Proposition \ref{pro:legendre duality for energies}
and section \ref{sub:The-relative-entropy}. This gives, just as above,
\[
\sup_{\mu\in E_{1}(X,\omega)}F(\mu)\leq\sup_{u\in C^{0}(X)}((\mathcal{E\circ}P)-\mathcal{L})(u)=\sup_{u\in\mathcal{H}(X,\omega)}(\mathcal{E}-\mathcal{L})(u)
\]
 which finishes the proof of the inequality in the first point of
the theorem. The fact that the sup over $E_{1}(X,\omega)$ may be
taken over the subspace of volume forms will be given in the proof
of the third point.

\emph{Proof of the second point:}

Let us first prove that if $u_{\mu}\in\mathcal{E}^{1}(X,\omega),$
then 
\begin{equation}
F(MA(u_{\mu}))\leq\mathcal{G}(u_{\mu})\label{eq:pf of fourth pt}
\end{equation}
 with equality iff $u$ solves equation \ref{eq:normalized me}. To
this end write $\mu:=MA(u_{\mu})=d\mathcal{E}_{|u_{\mu}}.$ Then,
by definition, 
\[
F(\mu)=\mathcal{E}(u_{\mu})-\left\langle u_{\mu},\mu\right\rangle -\sup_{u\in C^{0}(X)}(\mathcal{L}(u)-\left\langle u,\mu\right\rangle )\leq\mathcal{E}(u_{\mu})-\left\langle u_{\mu},\mu\right\rangle -(\mathcal{L}(u_{j})-\left\langle u_{j},\mu\right\rangle )
\]
 for any $u_{j}\in C^{0}(X).$ In particular, taking continuous functions
$u_{j}$ decreasing to $u_{\mu}$ and letting $j\rightarrow\infty$
and using the monotone convergence theorem proves the inequality \ref{eq:pf of fourth pt}.
Moreover, equality above clearly holds iff $u_{\mu}$ realizes the
sup defining $\mathcal{L}^{*}(\mu).$ By Lemma \ref{lem:sup of l attained}
this happens, since we assume that $\int e^{-u_{\mu}}\mu_{0}$ is
finite, iff 
\[
\mu=e^{-u_{\mu}}/\int e^{-u_{\mu}}\mu_{0}
\]
 which finishes the proof of the equality case in \ref{eq:pf of fourth pt}.

Next, to prove the inequality \ref{eq:lower bound on F by G} first
observe that, as explained above, setting $\mu':=e^{-u'}/\int e^{-u'}\mu_{0}$
with $u'\in\mathcal{E}^{1}(X,\omega)$ gives 
\[
F(\mu')=\sup_{u\in\mathcal{E}^{1}(X,\omega)}(\mathcal{E}(u)-\left\langle u,\mu'\right\rangle )-(\mathcal{L}(u')-\left\langle u',\mu'\right\rangle )\geq\mathcal{E}(u')-\mathcal{L}(u')=\mathcal{G}(u')
\]
 since $u'$ is a candidate for the sup. Moreover, by Theorem \ref{thm:var sol of ma}
equality holds iff $MA(u')=\mu'$ which means that $u'$ is a solution
of the  equation \ref{eq:normalized me}. As for the case when $n=1$
we take $u'$ continuous, but without assuming $\omega_{u'}\geq0.$
We can then repeat the same argument as above but taking the sup above
over $C^{0}(X)$ instead of $\mathcal{E}^{1}(X,\omega)$ (see remark
\ref{rem:legendre when n is one}).

Note that a necessery condition for $F_{\beta}$ to be bounded from
above is that $\mu_{0}$ has finite energy. Indeed, by Jensen's inequality
$\mathcal{E}_{\omega}(u)-\int_{X}u\mu_{0}\leq\mathcal{G}_{\beta}(u)$
which by the first theorem above is bounded from above and hence it
follows from the definition that $E(\mu_{0})<\infty.$ 
\end{proof}

\subsubsection{\label{sub:Intermezzo:-properness-vs}Intermezzo: properness vs coercivity}

Before continuing we we will briefly discuss some relations between
properness and coercivity of the functionals $\beta F_{\beta}$ and
$\beta G_{\beta}$ that will not be used elsewhere. It follows immediately
from inequality \ref{eq:f smal then g} above that if $\beta G_{\beta}$
is proper (wrt energy) then so is $\beta F_{\beta}.$ It would be
interesting to know if the converse is true? In the Kähler-Einstein
setting this was indeed shown by Tian, see \cite{ti} (Thm 7.13).
The proof is indirect and uses the continuity method to first establish
the existence of a Kähler-Einstein metric $\omega_{KE}$. Using the
existence of $\omega_{KE},$ reversing the continuity method and also
smoothing by the Kähler-Ricci flow then gives the properness of $\beta G_{\beta}$
(in this case $\beta=-1).$ As conjectured by Tian and subsequently
established in \cite{p-s+} the previous argument can be refined to
give that $\beta G_{\beta}$ is even coercive. All in all this in
particular shows that $\beta F_{\beta}$ is coercive iff $\beta G_{\beta}$
is. As next observed this latter property can be obtained as a corollary
of Theorem \ref{thm:f and g when beta neg} in the setting of a \emph{general}
measure $\mu_{0}:$ 
\begin{cor}
\label{cor:coerc}Let $\mu_{0}$ be a measure on $X$ and $\beta(=-\gamma)$
a negative number. Then the corresponding functional $\beta F_{\beta}$
is coercive iff $\beta G_{\beta}$ is coercive. \end{cor}
\begin{proof}
Assume that $\beta F_{\beta}$ is coercive or equivalently that $F_{\beta(1+\delta)}$
is bounded from above for some $\delta>0.$ Then it follows from Theorem
\ref{thm:f and g when beta neg} that $G_{\beta(1+\delta)}$ is also
bounded from above, i.e. for any $\omega-$psh function $v$ we have
\[
\frac{1}{\gamma(1+\delta)}\log\int e^{-\gamma(1+\delta)v}\mu_{0}\leq-\mathcal{E}_{\omega}(v)+C
\]
 To prove coercivity for $\beta G_{\beta}(u)$ we let $u$ be an arbitrary
$\omega-$psh function. By scale invariance it will be enough to consider
the case $\int u\omega^{n}=0,$ so that $-\mathcal{E}_{\omega}(u)=J_{\omega}(u).$
Then $v:=u/(1+\delta)$ is also $\omega-$psh function (since $\delta>0)$
such that $\int v\omega^{n}=0$ and hence applying the previous inequality
to $v$ gives 
\[
\frac{1}{\gamma}\log\int e^{-\gamma u}\mu_{0}-C\leq1+\delta)J_{\omega}(u/(1+\delta)\leq(1+\delta)^{-1/n}J_{\omega}(u),
\]
 where the last inequality follows from $J_{\omega_{0}}(tu)\leq t^{1+1/n}J_{\omega_{0}}(u)$
if $0<t<1$ (see remark 2 in \cite{din}). Since, $(1+\delta)^{-1/n}<1$
this shows that $-G_{-\gamma}$ is also coercive. The reversed implication
follows immediately from Theorem \ref{thm:f and g when beta neg}. 
\end{proof}

\subsubsection{A continuity lemma}

We will next prove a useful continuity result, using a minor modification
of the proof of Lemma 1.14 in \cite{b-b-w} (see also the proof of
the implication $(iii)\Longrightarrow(i)$ in Thm 3.1 in \cite{bbgz}).
\begin{lem}
\label{lem:kont of exp integral with bd}Assume that $\psi_{j}\rightarrow\psi$
in $PSH(X,\omega)$ (in the $L^{1}(X)-$topology) and that there is
a positive number $\delta$ such that 
\[
\int_{X}e^{-(\gamma+\delta)\psi_{j}}\mu_{0}\leq C
\]
 where the measure $\mu_{0}$ does not charge pluripolar sets. Then
\begin{equation}
\int_{X}e^{-\gamma\psi_{j}}\mu_{0}\rightarrow\int_{X}e^{-\gamma\psi}\mu_{0}\label{eq:conv of exp integral in lemma}
\end{equation}
 for any real number $\gamma.$ \end{lem}
\begin{proof}
Let $u_{j}:=e^{-\gamma\psi_{j}}$ and $u:=e^{-\gamma u}.$ By assumption
there is a constant $C$ and $p>1$ such $\left\Vert u_{j}\right\Vert _{L^{p}(\mu_{0})}\leq C.$
Hence, it follows from general functional analysis (using that the
unit ball in $L^{p}(\mu_{0})$ is weakly compact and the Hahn-Banach
separation theorem (compare the proof of Lemma 1.14 in \cite{b-b-w})
that there is a sequence $v_{j}$ of convex combinations of $u_{j}$
such that $v_{j}$ converges strongly to $v\in L^{p}(\mu_{0}).$ In
particular after replacing $u_{j}$ with any subsequence such the
the first integral in \ref{eq:pf of cont lemma} converges we get
\begin{equation}
\lim_{j}\int u_{j}\mu_{0}=\lim_{j}\int v_{j}\mu_{0}=\int v\mu_{0}\label{eq:pf of cont lemma}
\end{equation}
Since $\mu$ does not charge pluripolar sets Hartog's lemma \cite{g-z}
gives that $\limsup\psi_{j}=\psi$ a.e. wrt $\mu_{0},$ i.e. $\liminf u_{j}=u$
a.e. wrt $\mu_{0}$ so that $\liminf v_{j}\geq u$ a.e. wrt $\mu_{0}.$
But then \ref{eq:pf of cont lemma} and the $L^{p}(\mu_{0})-$ convergence
of $v_{j}$ forces $v=u$ a.e. wrt $\mu_{0}.$ According to \ref{eq:pf of cont lemma}
that ends the proof of the lemma.
\end{proof}

\subsubsection{Existence and convergence of maximizers for the free energy}

Next, we will prove one of the main results of the present paper showing
that coercivity of the   functional $F$ is sufficient for the existence
of a maximizer. 
\begin{thm}
\label{thm:proper gives max of F beta neg}Let $\beta=-\gamma<0.$
Suppose that the functional $-F_{-\gamma}$ is coercive (wrt energy)
or equivalently that $F_{-\gamma-\delta}$ is bounded for some $\delta>0,$
then $F_{-\gamma}$ admits a finite energy maximizer $\mu_{\beta}.$
Moreover, the potential of any maximizer solves the  equation \ref{eq:normalized me}.
More generally, if $\mu_{j}$ is a sequence such that 
\[
F_{-\gamma}(\mu_{j})\rightarrow\sup_{E_{1}(X,\omega)}F_{-\gamma}<\infty
\]
 Then, perhaps after passing to a subsequence, $\mu_{j}$ converges
weakly to a maximizer $\mu_{\beta}.$ If $\mu_{0}$ is a volume form
then the maximizer is smooth.\end{thm}
\begin{proof}
Let $\mu_{j}$ be a maximizing sequence for $F_{-\gamma},$ as in
the assumptions above. The boundedness assumption of $F_{-(\gamma+\delta)}$
is equivalent to the bound $\gamma F_{-\gamma}\leq-\delta E+C.$ Since,
by assumption, $F_{-\gamma}$ is bounded from below along $\mu_{j}$
it follows immediately that $E(\mu_{j})\leq C'.$ Writing $\mu_{j}=MA(u_{j})$
this means according to Lemma \ref{lem:energy as i minus j}, that
$(I-J)(u_{j})$ and hence $J(u_{j})$ are uniformly bounded: 
\[
J(u_{j})\leq C_{\delta}
\]
 Combining this latter bound with the fact that $G_{\gamma+\delta}$
is also bounded from above (by the first point in Theorem \ref{thm:f and g when beta neg})
gives that 
\begin{equation}
\int_{X}e^{-(\gamma+\delta)u_{j}}\mu_{0}\leq C_{\delta}\label{eq:pf of thm coerc}
\end{equation}
 and hence after adjusting by constants to get $\sup u_{j}=0$ and
passing to a subsequence to make sure that $u_{j}\rightarrow u$ in
$L^{1},$ the convergence \ref{eq:conv of exp integral in lemma}
in Lemma \ref{lem:kont of exp integral with bd} gives, also using
that $\mathcal{E}$ is usc (Prop \ref{pro:energy is ups}) 
\[
\infty>G_{-\gamma}(u)\geq\limsup G_{-\gamma}(u_{j})
\]
 Combining this with the first and second point in Theorem \ref{thm:f and g when beta neg}
gives 
\[
\sup_{\mathcal{E}^{1}(X,\omega)}G_{-\gamma}=\sup_{E_{1}(X,\omega)}F_{-\gamma}=\limsup F_{-\gamma}(\mu_{j})\leq\limsup\mathcal{G}_{-\gamma}(u_{j})\leq G_{-\gamma}(u)<\infty
\]
 and hence $u$ is a maximizer of $G_{-\gamma}$ on $\mathcal{E}^{1}(X,\omega).$
But then it follows precisely as in the proof of Theorem \ref{thm:existence etc when beta pos}
above, using the projection operator $P,$ that $u$ is a solution
of equation \ref{eq:normalized me}. 
\end{proof}

\subsection{The proof of Theorem \ref{thm:main intro} and a refined version}

Apart from the last statement in the theorem concerning properness
the proof follows immediately from combining the theorems established
above. Finally, in the general case when $F_{-\gamma}$ is only assumed
proper the previous proof still applies as long as $\mu_{0}$ satisfies
the following qualitative Moser-Trudinger type inequality: there is
a $\delta>0$ such that for any Kähler potential $u$ 
\begin{equation}
J(u)\leq C\implies\int e^{-(\gamma+\delta)(u-\sup u)}\mu_{0}\leq C'\label{eq:qual m-t ineq for measure}
\end{equation}
where $C$ depends on $\gamma,\delta$ and $C.$ This inequality does
hold in the case when $\mu_{0}=fdV$ with $f\in L^{p}(X,dV)$ for
$p>1$ as follows immediately from Hölder's inequality and the following
stronger property of any volume form $dV:$ 
\begin{equation}
J(u)\leq C\implies I_{t}(u):=\int e^{-t(u-\sup u)}dV\leq C_{t}.\label{eq:pf main in proper case}
\end{equation}
 for\emph{ any }$t>0$ obtained in the proof of Lemma 6.4 in \cite{bbgz},
using Zeriahi's uniform variant of Skoda's theorem \cite{ze}. More
generally, the previous arguments shows the that the following refined
version of the last part of Theorem \ref{thm:main intro} holds:
\begin{thm}
\label{thm:refined conv towards opt} Assume that $\mu_{0}$ satisfies
the qualitative Moser-Trudinger type inequality \ref{eq:qual m-t ineq for measure}
and let $u_{j}$ be a sequence in $\mathcal{E}^{1}(X,\omega)$ such
that $J_{\omega}(u_{j})\leq C$ (or equivalently, $\mathcal{E}_{\omega}(u-\sup u_{j})\geq-C').$
Then $F_{-\gamma}(u_{j})$ is uniformly bounded from above. If furthermore
$u_{j}$ is a maximizing sequence for $F_{-\gamma}$ then $u_{j}-\sup u_{j}$
converges (after perhaps passing to a subsquence) to a maximizer for
$F_{-\gamma}\circ MA.$ \end{thm}
\begin{rem}
It may be worth pointing out that the convergence $I_{t}(u_{j})\rightarrow I(u)$
used in the proof above can also be deduced from the results of Demailly-Kollar
\cite{d-j}. Indeed, since $J(u_{j})\leq C$ and we may assume that
$u_{j}\rightarrow u$ in $L^{1}(X),$ the fact that $\mathcal{E}$
is usc (and hence $J$ is lsc) gives $J(u)\leq C<\infty.$ But then
$u$ has no Lelong numbers (as follows from Cor 1.4 in \cite{g-z2})
and hence $I_{t}(u)<\infty$ for all $t$ (compare the proof of prop
\ref{pro:reg}). But then it follows from Theorem \ref{thm:(Demailly-Kollar-).-Let}
that $I_{t}(u)\rightarrow I(u)$ (compare the proof of Cor \ref{cor:global cont of exp int}).
\end{rem}

\subsection{Alpha-invariants}

We define the (generalized) alpha-invariant of a pair $([\omega],\mu_{0})$
by 
\[
\alpha([\omega],\mu_{0}):=\sup\left\{ \alpha:\exists C_{\alpha}:\,\int_{X}e^{-\alpha(u-\sup_{X}u)}\mu_{0}\leq C_{\alpha},\,\forall u\in PSH(X,\omega)\right\} 
\]
When $\mu_{0}$ is any given volume form on $X$ and the Kähler class
$[\omega]=c_{1}(L)$ is the first Chern class of an ample line bundle
the corresponding invariant of the class $[\omega]$ coincides with
the algebro-geometrically defined \emph{log canonical threshold of
$L$} \cite{dem2} (which is precisely Tian's original $\alpha-$variant
\cite{ti1} when $c_{1}(L)=-c_{1}(K_{X})).$ The case of a singular
measure $\mu_{0}$ was recently studied by Dinh-Nguyên-Sibony in complex
dynamics \cite{d-s}. In their terminology, $\alpha([\omega],\mu_{0})>0$
precisely when the measure $\mu_{0}$ is of \emph{global moderate
growth} (with respect to the Kähler class $[\omega]).$ As shown in
\cite{d-s} this condition in particular holds when $\mu_{0}=\omega_{u_{0}}^{n}/n!$
for an $\omega-$psh function $u_{0}$ which is Hölder continuous
and in particular for many of the equilibrium measures which arise
as limits in complex dynamics and whose supports typically are fractal
sets.
\begin{example}
\label{exa:fractal m-t}If $(X,\omega)$ is a Riemann surface with
$\int_{X}\omega=1$ then $\alpha([\omega],\omega)=1.$ Indeed, if
we denote by $G_{x_{0}}$ the corresponding \emph{Green function with
a pole at $x_{0}$} defined by $dd^{c}G_{x_{0}}=\delta_{x_{0}}-\omega$
and mean zero, where $\delta_{x_{0}}$ is the Dirac mass at the point
$x_{0}$ then the integral $\int_{X}e^{-\alpha(u-\sup_{X}u)}\mu_{0}$
for $u=G_{x_{0}}$ is finite for $\alpha<1$ and infinite for $\alpha=1$
(as follows from the standard fact that $G_{x_{0}}-\log d^{2}(x,x_{0})\in C^{0}(X)$
in terms of the distance function wrt the metric $\omega).$ Decomposing
a general element $u\in PSH(X,\omega)$ as $u(x)=\int u(y)G_{y}(x)\omega(y)$
and using Jensen's inequality then proves the claim. Similarly, if
there are positive constants $C$ and $d$ such that the measure $\mu_{0}$
satisfies 
\[
\mu(B_{r})\leq Cr^{d},
\]
 for $r$ sufficiently small, for every geodesic ball of radius $r,$
then 
\[
\alpha\geq2d.
\]
\end{example}
\begin{thm}
\label{thm:alpha invariant intro}Let $(X,\omega)$ be a compact Kähler
manifold and let $\mu_{0}$ be a probability measure on $X$ of finite
energy. If the parameter $\beta:=-\gamma$ (with $\gamma>0)$ satisfies
the bound 
\begin{equation}
\gamma<\alpha\frac{n+1}{n}\label{eq:cond on beta in thm alpha intro}
\end{equation}
 where $\alpha$ is the alpha-invariant of the pair $([\omega],\mu_{0}),$
then the following holds: 
\begin{itemize}
\item Both the functionals $F_{\beta}$ and $\mathcal{G}_{\beta}$ are bounded
from above, i.e. the corresponding\emph{ logarithmic Hardy-Sobolev}
and \emph{ Moser-Trudinger }type inequalities hold 
\item There is a maximizer $\mu$ of $F_{\beta}.$ Moreover, its potential
$u_{\mu}$ maximizes $\mathcal{G}_{\beta}$ and solves the  equation
\ref{eq:me intro}. 
\end{itemize}
\end{thm}
\begin{proof}
By Theorems \ref{thm:proper gives max of F beta neg} and \ref{thm:f and g when beta neg}
it will be enough to prove that $F_{-\gamma}$ is coercive under the
assumptions of the theorem. To this end first note that by assumption
we have that 
\[
\mathcal{L}_{t}^{-}(u)>-C
\]
 for any fixed $t$ with $t<\alpha.$ Writing $\mu=MA(u_{\mu})$ for
the potential $u_{\mu}$ such that $\sup u_{\mu}=0$ gives 
\[
\frac{1}{t}D(\mu)=\sup_{u}\mathcal{L}_{t}^{-}(u)-\left\langle u,\mu\right\rangle \geq\mathcal{L}_{t}^{-}(u_{\mu})-\left\langle u_{\mu},\mu\right\rangle \geq-\left\langle u_{\mu},\mu\right\rangle -C,
\]
 i.e. 
\[
D(\mu)\geq-t\left\langle u_{\mu},\mu\right\rangle -C
\]
 This means that 
\[
F_{-\gamma}(\mu)\leq E(\mu)+\frac{t}{\gamma}\left\langle u_{\mu},\mu\right\rangle +C
\]
 Combining the previous inequality with the inequality \ref{eq:energy bounded by pairing}
hence gives 
\[
F_{-\gamma}(\mu)\leq E(\mu)(1-\frac{t}{\gamma}(\frac{n+1}{n}))+C,
\]
 showing that $F_{-\gamma}$ is proper and even coercive (wrt energy)
as long as 
\[
\gamma<\alpha(\frac{n+1}{n}))
\]
 and $t$ is chosen sufficiently close to $\alpha.$Hence the theorem
follows from Theorem \ref{thm:main intro}.
\end{proof}
In particular, specializing to a Riemann surfaces with $\mu_{0}$
a Frostman measure gives the following 
\begin{cor}
\label{cor:m-t on rieman sur intro}Let $X$ be a compact Riemann
surface and $\mu_{0}$ a probability measure such that 
\[
\mu_{0}(B_{r})\leq Cr^{d}
\]
 for some positive constants $C$ and $d,$ for any local coordinate
ball $B_{r}$ of sufficiently small radius $r.$ Then, for any $\delta>0$
there is a constant $C_{\delta}$ such that 
\[
\log\int_{X}e^{u}\mu_{0}\leq\frac{(d+\delta)}{2}\frac{1}{4}\int_{X}du\wedge d^{c}u+C_{\delta}
\]
 for any smooth function $u$ on $X$ normalized so that $\int_{X}u\omega=0$
for a fixed measure $\omega$ on $X.$ \end{cor}
\begin{proof}
Let us first prove that when $n=1$ the bound on $\mathcal{G}_{-\gamma}(v)$
in fact holds for all smooth functions $v$ on $X.$ This can be seen
in two ways. First, it follows precisely as in the proof Cor 3 in
\cite{b0} from using the following inequality for $v\in\mathcal{C}^{\infty}(X)$
proved there: 
\[
\mathcal{E}_{\omega}(P_{\omega}v)\geq\mathcal{E}_{\omega}(v)
\]
 (which is a rather direct consequence of the orthogonality relation
\ref{eq:orthog relation} when $n=1).$ Combining the previous inequality
with the fact that $\mathcal{L}_{\gamma}^{-}(u)$ is increasing in
$u$ immediately gives 
\[
\sup_{v\in\mathcal{C}^{\infty}(X)}\mathcal{G}_{-\gamma}(v)\leq\sup_{v\in\mathcal{C}^{\infty}(X)}\mathcal{G}_{-\gamma}(P_{\omega}v)\leq\sup_{\mathcal{H}(X,\omega)}\mathcal{G}_{-\gamma}(v)
\]
 which is bounded by Theorem \ref{thm:alpha invariant intro}. Alternatively,
for $v$ continuous we let $\mu:=e^{-\gamma v}/\int e^{-\gamma.v}\mu_{0}.$
Then, by the last point in Theorem \ref{thm:f and g when beta neg}
\[
\mathcal{G}_{-\gamma}(v)\leq F_{\gamma}(e^{-\gamma v}/\int e^{-\gamma v}\mu_{0})\leq C
\]
 using Theorem \ref{thm:alpha invariant intro} in the last step (in
the Kähler-Einstein setting on $S^{2}$ a similar argument was used
by Rubinstein \cite{rub0}). Finally, since if $\int u\omega=0$ and
$\int\omega=1$ we have 
\[
\mathcal{E}_{\omega}(u)=-\frac{1}{2}\int_{X}du\wedge d^{c}u
\]
 and hence $\mathcal{E}_{\omega}(cu)=c^{2}\mathcal{E}_{\omega}(u).$
All in all this means that we obtain the inequality we wanted to prove
from $\mathcal{G}_{-\gamma}(\frac{1}{\gamma}u)\leq C$ 
\end{proof}
It seems likely that one can take $\delta=0$ in the previous corollary
by further studying the blow-up behavior of the functional \emph{$\mathcal{G}_{\alpha-\delta}$}
when $\delta\rightarrow0.$ Indeed, when $\mu$ is a volume (are rather
area) form setting $\delta=0$ does give an optimal inequality according
to Fontana's generalization \cite{cher} of Moser's inequality on
the two-sphere $S^{2}.$ Even though formulated for Riemann surfaces
without boundary the corollary above also contains the analogous statement
on any compact Riemann surface $Y$ with smooth boundary $\partial Y$
if one demands, as usual, that $u=0$ on $\partial Y.$ Indeed, if
$Y$ is a domain in the compact closed Riemann surface $X$ and $u\in C^{0}(Y)$
with $y=0,$ or more generally $u$ is in the Sobolev space $H_{0}^{1}(Y)$
(i.e. the closure in the Dirichlet norm of the space $C_{0}^{\infty}(Y)$
of all smooth and compactly supported functions on the interior of
$Y)$ it is, by standard continuity arguments, enough to prove the
inequality for $u\in C_{0}^{\infty}(Y).$ Extending by zero gives
$u\in\mathcal{C}^{\infty}(X)$ and then the inequality then follows
immediately from Corollary \ref{cor:m-t on rieman sur intro} when
$\omega$ is taken as a measure supported on $\partial Y$ in $X.$ 

In particularly, taking $Y$ as a domain in $\R^{2}$ one gets a weak
version of a recent result och Cianchi \cite{ci} who proved the corresponding
inequality with $\delta=0,$ using completely different methods. This
latter result has very recently been further developed, still in the
setting of $\R^{2},$ by Morpurgo-Fontana \cite{f-m}, building on
Adam's seminal work.%
\footnote{It was pointed out in \cite{f-m} that the methods in \cite{f-m}
can be generalized to the setting of compact manifolds using pseudo-differential
calculus - presumably such a generalization would lead to the sharp
version of Cor \ref{cor:m-t on rieman sur intro} discussed above.
Moreover, the results in \cite{f-m} also give higher dimensional
Moser-Trudinger type inequalities, but for other operators than the
Monge-Ampère operator.%
}

\subsection{\label{sub:The-case-when beta is inf}The limit $\beta\rightarrow\infty:$
envelopes and free boundaries }

In this section we will take the fixed form $\omega$ on $X$ to be
any smooth and closed $(1,1)-$form defining a Kähler\emph{ class}
in $H^{1,1}(X,\R)$ (but not necessarily a Kähler \emph{form}). Consider
the following free boundary value problem for a function $u$ on $X:$
\begin{equation}
\begin{array}{rcll}
(\omega+dd^{c}u)^{n} & = & 0 & \mbox{on\,\ensuremath{\left\{ u<0\right\} }}\\
u & \leq & 0 & \mbox{on}\, X\\
\omega_{u} & \geq & 0 & \mbox{on}\, X
\end{array}\,\label{eq:free boundary}
\end{equation}
 It follows immediately from the domination principle for the Monge-Ampère
operator (see Cor 2.5 in \cite{begz}) that the solution is unique
and can be represented as an upper envelope: 
\begin{equation}
P_{\omega}0=\sup_{v\in PSH(X,\omega)}\{v(x):\, v\leq0\,\mbox{on\,}X\}\label{eq:envelope in beta inf}
\end{equation}

\begin{thm}
\label{thm:free bound}Given a volume form $\mu_{0}$ on $X$ and
$\beta>0$ let $v_{\beta}$ the unique solution of the non-normalized
 equation \ref{eq:normalized me} and $u_{\beta}$ the unique solution
of equation \ref{eq:non-normalized} normalized so that $\sup_{X}u_{\beta}=0.$
Then both $u_{\beta}$ and $v_{\beta}$ converge in $L^{1}(X)$ to
a the solution of the free boundary value problem \ref{eq:free boundary},
which in turn coincides with the envelope $P_{\omega}0$ above. \end{thm}
\begin{proof}
Let $\mathcal{L}_{\beta}^{+}(u):=\frac{1}{\beta}\log\int_{X}e^{\beta u_{\beta}}\mu_{0}.$

\emph{Step 1: Convergence of $u_{\beta}$}

Since $\mu:=MA(P_{\omega}0)$ is a candidate for the sup defining
the Legendre transform of $D_{\mu_{0}}$ we get (see section \ref{sub:The-relative-entropy}
or use directly Jensen's inequality) 
\[
-\int_{X}u_{\beta}MA(P_{\omega}0)+\frac{1}{\beta}D_{\mu_{0}}(MA(P_{\omega}0))\geq-\frac{1}{\beta}\mathcal{L}_{\beta}^{+}(u_{\beta})
\]
 Hence defining the constant $D:=D_{\mu_{0}}(MA(P_{\omega}0))$ gives
\begin{equation}
\mathcal{E_{\omega}}(u_{\beta})-\int_{X}u_{\beta}MA(P_{\omega}0)+\frac{D}{\beta}\geq\mathcal{E_{\omega}}(u_{\beta})-\frac{1}{\beta}\mathcal{L}_{\beta}^{+}(u_{\beta})\geq\label{eq:pf of thm beta inf ineq}
\end{equation}
 
\[
\geq\mathcal{E_{\omega}}(P_{\omega}0)-\mathcal{L}_{\beta}^{+}(P_{\omega}0)
\]
 using, in the last inequality that, by Theorem \ref{thm:existence etc when beta pos}b
$u_{\beta}$ maximizes the functional $\mathcal{G}_{\beta}.$ Since
\[
\mathcal{L}_{\beta}^{+}(P_{\omega}0)\rightarrow\sup_{X}P_{\omega}0=0
\]
 (the last equality above follows for example from the orthogonality
relation \ref{eq:orthog relation}) this means that 
\begin{equation}
\liminf_{\beta\rightarrow\infty}\mathcal{E_{\omega}}(u_{\beta})-\int_{X}(u_{\beta}MA(P_{\omega}0)\geq\mathcal{E_{\omega}}(P_{\omega}0)-\int_{X}(P_{\omega}0))MA(P_{\omega}0)\label{eq:pf of thm beta inf}
\end{equation}
 also using the orthogonality relation \ref{eq:orthog relation} saying
that the second term in the rhs vanishes. But by the last statement
in Theorem \ref{thm:var sol of ma} it then follows that $u_{\beta}\rightarrow P_{\omega}0$
in $L^{1}(X)$ and that \ref{eq:pf of thm beta inf} is actually an
equality when $\liminf$ is replaced by $\lim.$

\emph{Step two: Convergence of $v_{\beta}$}

By the asymptotic equality referred to above combined with the fact
that $u_{\beta}\rightarrow P_{\omega}0$ and the orthogonality relation
we get the following ``convergence in energy'': 
\[
\mathcal{E_{\omega}}(u_{\beta})\rightarrow\mathcal{E_{\omega}}(P_{\omega}0)
\]
 Hence, using the orthogonality relation \ref{eq:orthog relation}
again the inequalities \ref{eq:pf of thm beta inf ineq} force 
\[
-\frac{1}{\beta}\mathcal{L}_{\beta}^{+}(u_{\beta})\rightarrow0
\]
 i.e. $v_{\beta}:=u_{\beta}-\frac{1}{\beta}\mathcal{L}_{\beta}^{+}(u_{\beta})$
has the same limit as $u_{\beta}$ and satisfies the equation \ref{eq:non-normalized}. 
\end{proof}
As shown in \cite{b-d} the envelope $P_{\omega}0$ has a Laplacian
which locally bounded it hence seems natural to ask if the convergence
above holds in the Hölder space $C^{1,\alpha}(X)$ for any $\alpha<1$?

\section{\label{sub:The-(twisted)-K=00003D0000E4hler-Einstein}The (twisted)
Kähler-Einstein setting}

In this section the measure $\mu_{0}$ will be taken to be a volume
form and we will then reformulate equation \ref{eq:me intro} as a
twisted Kähler-Einstein equation. First recall that the\emph{ Ricci
curvature} of a Kähler metric is defined, in local holomorphic coordinates,
by 
\[
\mbox{Ric}\omega:=dd^{c}(-\log(\frac{\omega^{n}}{(i\sum_{j}dz_{j}\wedge d\bar{z}_{j})^{n}}))(=-dd^{c}\log(\det\omega_{ij}))
\]
 representing the anti-canonical class $-c_{1}(K_{X}).$ If $\theta$
is a given closed $(1,1)-$form on $X$ the \emph{twisted Kähler-Einstein
equation} for a Kähler metric $\omega$ is defined by 
\begin{equation}
\mbox{Ric}\omega-\theta=-\beta\omega\,\,\:(\gamma:=-\beta\in\R)\label{eq:twisted k-e eq}
\end{equation}
 where, compared with the previous notation and the lhs is called
the\emph{ twisted Ricci curvature} of $\omega.$ It hence implies
the following cohomological relation in $H_{dd^{c}}^{2}(X,\R)$: 
\begin{equation}
[\omega]=\beta(c_{1}(K_{X})+[\theta])\label{eq:cohom relation}
\end{equation}
 forcing $\beta(c_{1}(K_{X})+[\theta])$ to be a Kähler class, which
we will henceforth assume. Fixing a Kähler form $\omega=\omega_{0}$
in $\beta(c_{1}(K_{X})+[\theta]),$ one defines its \emph{twisted
Ricci potential }$h=h_{\omega,\theta}$ by the following equation
\begin{equation}
\mbox{Ric}\omega-\theta=-\beta(\omega+dd^{c}h_{\omega,\theta}),\label{eq:twisted ricci pot}
\end{equation}
 where the normalization constant is fixed by imposing $\int_{X}e^{-h_{\omega,\theta}}\omega^{n}=1.$
Then \ref{eq:twisted k-e eq} (with $\omega=\omega_{u})$ is equivalent
to the equation 
\begin{equation}
(\omega+dd^{c}u)^{n}=e^{-\beta h_{\omega,\theta}}e^{\beta u}\omega^{n},\label{eq:m-e equation in twisted set}
\end{equation}
 i.e. the  equation \ref{eq:non-normalized} with $[\omega]$ satisfying
\ref{eq:cohom relation} and 
\begin{equation}
\mu_{0}=e^{-\beta h_{\omega,\theta}}\frac{\omega^{n}}{Vn!}\label{eq:mu zero in twisted setting}
\end{equation}
 We will call this particular choice of a triple $(\beta,\omega,\mu_{0})$
for the\emph{ twisted Kähler-Einstein setting. }In fact, the previous
argument shows that the equation \ref{eq:non-normalized} is equivalent
to the twisted Kähler-Einstein equation when $\mu_{0}$ is a volume
form, as follows by first defining $h_{\omega,\theta}$ by the relation
\ref{eq:mu zero in twisted setting} and then $\theta$ by the relation
\ref{eq:twisted ricci pot}.

\subsection{\label{sub:The-free-energy}The twisted Mabuchi K-energy functional
as the free energy}

Next, we define, for a fixed $\beta,$ $\mathcal{K}_{\theta}(u):=\beta F_{\beta}(MA(u)).$ 
\begin{prop}
\label{pro:var def of k}The functional $\mathcal{K}_{\theta}(u_{t})$
satisfies 
\begin{equation}
d\mathcal{K_{\theta}}_{|u}=(\beta\omega_{u}-\mbox{(Ric \ensuremath{\omega_{u}-\theta}}))\wedge\frac{\omega_{u}^{n-1}}{(n-1)!}\label{eq:diff of tw mab}
\end{equation}
 and $\mathcal{K}_{\theta}$ can hence be decomposed as $\mathcal{K}_{\theta}=\mathcal{K}^{(\beta)}+\mathcal{J}_{\theta}$
where 
\[
d\mathcal{K}_{|u}^{(\beta)}=(\beta\omega_{u}-\mbox{Ric \ensuremath{\omega_{u})\wedge\frac{\omega_{u}^{n-1}}{(n-1)!},\,\,\, d\mathcal{J}_{\theta|u}=\theta}}\wedge\frac{\omega_{u}^{n-1}}{(n-1)!}
\]
 \end{prop}
\begin{proof}
Combining Proposition \ref{pro:deriv of E} and \ref{pro:deriv of D}
gives
\[
\frac{d\mathcal{K}_{\theta}(u_{t})}{dt}=\int(-\beta u_{MA(u_{t})}+\log(\frac{MA(u_{t})}{\mu_{0}})\frac{dMA(u_{t})}{dt}
\]
 Now $\frac{dMA(u_{t})}{dt}=dd^{c}(\frac{du_{t}}{dt})\wedge\omega_{u_{t}}^{n-1}/(n-1)!$
and hence integration by parts give 
\[
\frac{d\mathcal{K}_{\theta}(u_{t})}{dt}=\int\frac{du_{t}}{dt}dd^{c}(-\beta u_{MA(u_{t})}+\log(\frac{MA(u_{t})}{\mu_{0}}))\wedge\omega_{u_{t}}^{n-1}/(n-1)!=
\]
 
\[
=\int\frac{du_{t}}{dt}(-\beta\omega_{u}+(\beta\omega+dd^{c}\log(\frac{MA(u_{t})}{\mu_{0}}))\wedge\omega_{u_{t}}^{n-1}/(n-1)!
\]
 using that, by definition, $\omega_{u_{MA(u_{t})}}=\omega_{u_{t}}.$
Finally, since the second term in the sum above may be written as
$(\beta\omega)_{\log(\frac{MA(u_{t})}{\mu_{0}})}=\mbox{-Ric}\ensuremath{\omega_{t}}+\theta$
when $\mu_{0}=e^{\beta h_{\omega,\theta}}\omega^{n}/Vn!$ this proves
the formula above for $d\mathcal{K_{\theta}}.$ 
\end{proof}
The previous proposition confirms that $\mathcal{K}_{\theta}(u)$
indeed coincides with Mabuchi's K-energy functional for $\theta=0$
and $\beta=1$ \cite{m2} and in general with its twisted versions
\cite{s-t,st} which are usually defined by the property \ref{eq:diff of tw mab}.
In the smooth setting the decomposition \ref{eq:def of f intro} is
then equivalent to a formula for \emph{$\mathcal{K}$} due to Tian
(see (5.12) in \cite{ti3}). Tian's formula was generalized by Chen
\cite{ch2} who used it to define and study $\mathcal{K}$ on potentials
$u$ such that $\omega_{u}$ is locally bounded. As emphasized in
the present paper formula \ref{eq:mab f intro} allows one to extend
the definition of $\mathcal{K}$ to the space $\mathcal{E}^{1}(X,\omega)$
of finite energy potentials. 
\begin{rem}
To compare with other formulations of the (twisted) Mabuchi functional
in the setting of log pairs we set $\beta=1$ and take $\theta$ to
be the current of integration along a smooth divisor $D.$ Writing
$\mathcal{K}_{(X,D)}:=\mathcal{K}_{\theta}$ we can then, trivially,
rewrite the relation \ref{eq:diff of tw mab} as 
\[
d\mathcal{K}_{(X,D)|_{u}}=-\left(\mbox{Ric \ensuremath{\omega_{u}\wedge\frac{\omega_{u}^{n-1}}{(n-1)!}-n\frac{(-K_{X})\cdot L^{n-1}}{L^{n}}\frac{\omega_{u}^{n}}{n!}}}\right)+\left(\delta_{D}\wedge\frac{\omega_{u}^{n-1}}{(n-1)!}-n\frac{D\cdot L^{n-1}}{L^{n}}\frac{\omega_{u}^{n}}{n!}\right),
\]
 using that, by definition, $L:=-(K_{X}+D).$ The first term is equal
to $-\omega_{u}^{n}/n!$ times $R-\bar{R},$ where $R$ is the scalar
curvature of the Kähler metric $\omega_{u}$ and $\bar{R}$ is its
average. Hence, up to an additive constant, $\mathcal{K}_{(X,D)}=\mathcal{K}+(\mathcal{E}_{(D,\omega)}-n\frac{D\cdot L^{n-1}}{L^{n}}\mathcal{E}_{(X,\omega)}),$
where $\mathcal{K}$ is the usual Mabuchi functional attached to the
Kähler class $[\omega]$ and $\mathcal{E}_{(X,\omega)}$ amd $\mathcal{E}_{(D,\omega)}$
are the usual energy functionals on $X$ and the submanifold $D$
defined as in section \ref{sub:The-Monge-Amp=00003D0000E8re-operator}.
\end{rem}
As shown by Mabuchi \cite{m1} and Donaldson $\mathcal{K}$ is convex
along \emph{geodesics} in $\mathcal{H}_{\omega}(X)$ (defined in terms
of \emph{Mabuchi's Riemannian metric} $g$ on $\mathcal{H}(X,\omega);$
see section \ref{sec:Convergence-of-the-cal} below). Using this latter
convexity we also deduce the following proposition. Before stating
it we recall that any complex curve $u_{t}$ in $\mathcal{H}_{\omega}(X)$
determines a curve $V_{t}$ of $(1,0)-$vector fields which are dual
to the $(0,1)-$form $\bar{\partial}(\partial_{\bar{t}}u)$ under
$\omega_{u_{t}}.$ 
\begin{prop}
\label{pro:convexity of twisted mab}If $\theta\geq0$ is a positive
current then the functional $\mathcal{K}_{\theta}(u_{t})$ is convex
along geodesics $u_{t}$ in $\mathcal{H}_{\omega}(X)$ and strictly
convex if $\theta$ is a Kähler current, i.e. $\theta>\epsilon\omega_{0}.$
Moreover, if $\theta$ is a positive multiple of the current of integration
$\delta_{D}$ along an irreducible smooth divisor $D,$ then $d^{2}\mathcal{K}_{\theta}(u_{t})/d^{2}t=0$
at a given $t$ iff $\bar{\partial}V_{t}=0$ and $V_{t}$ is tangential
to $D.$ In particular, $d^{2}\mathcal{K}_{\theta}(u_{t})$ is geodesically
strictly convex if $X$ admits no non-trivial holomorphic vector fields
which are tangent to $D.$\end{prop}
\begin{proof}
The first part was already observed by Stoppa \cite{st} and hence
we consider the case when $\theta=c\delta_{D}$ (and it will be clear
that we may assume that $c=1).$ Let us first recall the following
formula for a geodesic $u_{t}:$ 
\begin{equation}
\partial_{t}^{2}u_{t}-|\bar{\partial}(\partial_{\bar{t}}u)|_{\omega_{u_{t}}}^{2}(=\partial_{t}^{2}u_{t}-|V_{t})|_{\omega_{u_{t}}}^{2})=0,\label{eq:geod equ}
\end{equation}
 We also recall the following formula \cite{m1,d00} of the usual
Mabuchi functional along a geodesic (recall also that $\mathcal{E}_{\omega}$
is affine alongs geodesics): 
\[
\frac{\partial^{2}\mathcal{K}(u_{t})}{\partial^{2}t}=\int_{X}|\bar{\partial}V|_{\omega_{u_{t}}}^{2}\frac{\omega_{u_{t}}^{n}}{n!}(\geq0)
\]
 Next, a direct calculation gives 
\[
\frac{\partial^{2}\mathcal{J}_{\theta}(u_{t})}{\partial^{2}t}=\int_{D}(\partial_{t}^{2}u_{t}-|\bar{\partial}_{D}(\partial_{\bar{t}}u)|_{\omega_{u_{t}}}^{2})\frac{\omega_{u_{t}}^{n-1}}{(n-1)!}=\int_{D}|V_{N}|_{\omega_{u_{t}}}^{2}\frac{\omega_{u_{t}}^{n-1}}{(n-1)!}(\geq0)
\]
 where $V_{N}$ denotes the component of $V_{t}$ normal to $D$ wrt
$\omega_{t}$ and where we have used the geodesic equation \ref{eq:geod equ}
in the last step. The proof is now concluded by invoking the decomposition
formula for $\mathcal{K}_{\theta}$ from the previous proposition. 
\end{proof}
As shown by Bando-Mabuchi \cite{b-m} any Kähler-Einstein metric minimizes
$\mathcal{K}_{\theta}.$ Here we note that the corresponding property
holds in the (possibly singular) twisted setting for any \emph{positive}
current $\theta:$ 
\begin{prop}
\label{pro:minim of k-ene for pos theta}Let $\theta\geq0$ be a positive
current and $u\in\mathcal{E}^{1}(X,\omega)$ a solution to equation
\ref{eq:m-e equation in twisted set}, Then $u$ minimizes the functional
$\mathcal{K}_{\theta}$ on $\mathcal{E}(X.\omega).$ \end{prop}
\begin{proof}
By Theorem \ref{thm:f and g when beta neg} it will be enough to prove
that $u$ minimizes the corresponding twisted Ding functional $-\mathcal{G}_{\theta}.$
But $-\mathcal{G}_{\theta}$ is convex along $C^{0}-$geodesics \cite{bern}
and hence it is minimized on any critical point $u.$ 
\end{proof}
In the case when $\theta\geq0$ is smooth Stoppa \cite{st} deduced
the previous proposition from the geodesic convexity of $\mathcal{K}_{\theta},$
combined with the deep regularity theory for $C^{1,1}-$geodesics
of Chen-Tian (in the more general setting of twisted constant scalar
curvature metrics).

\subsection{\label{sub:Proof-of-Corollary} Alpha-invariants and Nadel sheaves}

In the twisted Kähler-Einstien setting we get the following refinement
of Theorem \ref{thm:alpha invariant intro}: 
\begin{thm}
\label{cor:alpha inv for twisted}Let $\gamma$ be a positive number
and $\theta$ a closed $(1,1)-$form on the $n-$dimensional compact
complex manifold $X$ such that the class $-(\gamma c_{1}(K_{X})+[\theta])$
in $H^{2}(X,\R)$ is Kähler (i.e. contains some Kähler form) 
\begin{itemize}
\item If the alpha-invariant of the class \textup{$-(\gamma c_{1}(K_{X})+[\theta])$}
satisfies 
\[
\alpha>\gamma\frac{n}{n+1}
\]
 then the class contains a Kähler form $\omega$ which solves the
twisted Kähler-Einstein equation 
\begin{equation}
\mbox{Ric}\omega=\gamma\omega+\theta\label{eq:twisted k-e eq in thm}
\end{equation}
 and which minimizes the twisted Mabuchi K-energy $\mathcal{K}_{\theta}.$ 
\item More precisely, if $u_{j}$ is a normalized asymptotically minimizing
sequence for $\mathcal{K}_{\theta}$ then any given $L^{1}-$accumulation
point $u_{\infty}$ of $u_{j}$ is \emph{either} the potential of
a $\theta-$twisted Kähler-Einstein metric\emph{ or }$u_{\infty}$
defines a Nadel type multiplier ideal sheaf, i.e. $\int_{X}e^{-t\gamma u_{\infty}}dV=\infty$
for any $t>\frac{n}{n+1}.$ 
\end{itemize}
\end{thm}
The parameter $\gamma$ may, of course, be set to one after scaling
$\omega$ but it has been included for later convenience. In the standard
un-twisted case, i.e. when $\theta=0$ the first point in the previous
corollary is due to Tian \cite{ti1}, who used the continuity method,
which as explained above is not applicable in the general twisted
setting. As for the second point above it generalizes a result of
Nadel \cite{na} and Demailly-Kollar \cite{d-j} concerning the case
when $u_{j}$ is a subsequence of the curve $u_{t}$ appearing in
the continuity method (see remark \ref{rem:nadel seq}) and hence
the result in the second point above is new even when $\theta=0$
and it implies the second point in Cor \ref{cor:min of mab intro}.
Indeed, when $X$ is Fano with no non-trivial holomorphic vector fields
it is well-known that there exists a (unique) Kähler-Einstein metric
iff $\mathcal{K}$ is proper (see section \ref{sub:Intermezzo:-properness-vs}).
Hence, either (1) $X$ admits a Kähler-Einstein metric and the convergence
in Cor \ref{cor:min of mab intro} then follows from Theorem \ref{thm:main intro}
or (2) it does not and then one applies Theorem \ref{cor:alpha inv for twisted}.

Twisted Kähler-Einstein metrics and the corresponding twisted Mabuchi
K-energy recently appeared in the works of Fine \cite{fine} and Song-Tian
\cite{s-t} (see also \cite{st} for relations to stability). Note
that for a twisting form $\theta$ which is not semi-positive the
minimizing property of the solution furnished by the Theorem above
is not automatic and moreover there are no uniqueness properties of
the solutions (see the discussion and references on p. 65 in \cite{ta}
for the Riemann surface case).

\subsubsection{Proof of. Theorem \ref{cor:alpha inv for twisted}}

The first point of the corollary is a direct consequence of Theorem
\ref{thm:alpha invariant intro} applied to the twisted Kähler-Einstein
setting. Next, we show how the proof can be refined so as to give
a proof of the second point in the corollary. After scaling we may
assume that $\gamma=1.$ Let $u_{j}$ be an asymptotic minimizinf
sequence for $\mathcal{K}_{\theta}$ such that $u_{j}\rightarrow u_{\infty}$
in $L^{1}$ (by weak compactness such an $u_{\infty}$ always exists).
If the second alternative in the statement of Cor \ref{cor:alpha inv for twisted}
does not hold then there is $t>\frac{n}{n+1}$ such that $\int e^{-tu_{\infty}}dV<\infty.$
But then it follows from the semi-continuity result of Demailly-Kollar
\cite{d-j} (see Thm \ref{thm:(Demailly-Kollar-).-Let} in the appendix)
that $\int e^{-tu_{j}}dV\leq C<\infty$ after perhaps replacing $t$
with any strictly smaller number. In the notation of the proof of
Thm \ref{thm:alpha invariant intro} this means that $\mathcal{L}_{t}^{-}(u_{j})>-C'$
and hence repeating that proof word for word shows that 
\begin{equation}
\mathcal{K}_{\theta}(u_{j})\geq J(u_{j})/C-C\label{eq:proof of cor nadel}
\end{equation}
 for some constant $C.$ Finally Theorem \ref{thm:refined conv towards opt}
shows that $u$ is a minimizer for $\mathcal{K}_{\theta}$ and satisfies
the twisted Kähler-Einstein equation. 
\begin{rem}
\label{rem:nadel seq}The second point in Cor \ref{cor:alpha inv for twisted}
generalizes Nadel's original result \cite{na}; letting $T$ be the
sup over all positive $t$ such that the equations appearing in Aubin's
continuity method have a solution $\omega_{t}:$ 
\begin{equation}
\mbox{Ric\ensuremath{\omega_{t}=t\omega_{t}+(1-t)\omega},}\label{eq:cont equation}
\end{equation}
Nadel shows (see also the simplifications in \cite{d-j}) that either
$T\ge1$ and the potential $u_{t}$ of $\omega_{t}$ converges to
a Kähler-Einstein metric or there is sequence $t_{j}\rightarrow T$
such $u_{t_{j}}\rightarrow u_{T}$ for $u_{T}(=u_{\infty})$ as in
the second point of Cor \ref{cor:alpha inv for twisted}. To see that
this is a special case of Cor \ref{cor:alpha inv for twisted} we
argue as above; if the second alternative does not hold then one checks
that $u_{t}$ is an asymptotic minimizing sequence for $\mathcal{K}_{\theta_{T}}$where
$\theta_{t}:=(1-t)\omega$ (see below) and hence we may apply the
second point in Cor \ref{cor:min of mab intro} (with $t=\gamma\leq1$
and $\theta=\theta_{T})$ to deduce that $u_{t_{j}}\rightarrow u_{T},$
where $\omega_{u_{T}}$ solves the twisted Kähler-Einstein equation
for $\theta=\theta_{T}.$ But then it follows from the definition
of $T$ that $T\geq1$ and hence $\omega_{u_{1}}$ is a Kähler-Einstein
metric proving Nadel's result. Finally, the asymptotic minimizing
property above is shown as follows: as is well-known $\mathcal{K}_{0}(u_{t})$
is decreasing in $t$ and hence $J(u_{t})\leq C$ (by \ref{eq:proof of cor nadel}).
But since $\theta_{t}\geq0$ $u_{t}$ is the absolute minimizer of
$\mathcal{K}_{\theta_{t}}$ (see the end of Remark \ref{rem:twisted calabi flow})
one deduces (also using $J(u_{t})\leq C)$ the desired asymptotic
minimizing property (by the same argument used in Step 2 in the proof
of Cor \ref{cor:alpha for pairs gives ec}). 
\end{rem}

\section{\label{sec:Convergence-of-the-cal}Convergence of the Calabi flow }

In this section we consider for simplicity the un-twisted case, i.e.
$\theta=0$ (see Remark \ref{rem:twisted calabi flow} below for the
twisted case). First recall that the \emph{Mabuchi metric} $g$ on
$\mathcal{H}(X,\omega)$ is defined by first identifying the tangent
space of $\mathcal{H}(X,\omega)\subset\mathcal{C}^{\infty}(X)$ at
the point $u$ with $\mathcal{C}^{\infty}(X)$ and then letting 
\[
g(v,v)_{|u}:=\int_{X}v^{2}(\omega_{u})^{n}/n!
\]
 We denote by $d(\cdot,\cdot)$ the corresponding distance function
on $\mathcal{H}(X,\omega).$ It follows directly from the variational
definition of the Mabuchi's K-energy functional $\mathcal{K}$ (see
Proposition \ref{pro:var def of k}) that its gradient on $(\mathcal{H}(X,\omega),g)$
is given by 
\[
\nabla\mathcal{K}_{|u}=-(R_{\omega_{u}}-R),
\]
 where $R_{\omega_{u}}$ denotes the scalar curvature of the Kähler
metric $\omega_{u}$ and $R$ its average, which is an invariant of
the class $[\omega].$ The \emph{Calabi functional} on $\mathcal{H}(X,\omega)$
may be defined as the squared norm of $\nabla\mathcal{K},$ i.e. 
\[
Ca(u):=\int_{X}(R_{\omega_{u}}-R)^{2}\omega_{u}^{n}/n!,
\]
We let $u_{t}$ evolve according to the Calabi flow on the level of
Kähler potentials, i.e. 
\begin{equation}
\frac{du_{t}}{dt}=(R_{\omega_{u_{t}}}-R)(=-\nabla\mathcal{K}_{|u_{t}})\label{eq:calabi flow on potentials}
\end{equation}
Before turning to the proof of Theorem \ref{thm:conv of cal intro}
we recall the result of Tian \cite{ti} saying that if $H^{0}(TX)=\{0\}$
then $X$ admits a Kähler-Einstein metric iff the functional $\mathcal{K}$
is proper (wrt energy); compare section \ref{sub:Intermezzo:-properness-vs}.
By Cor \ref{cor:min of mab intro} and the uniqueness of the Kähler-Einstein
metric under the assumptions above \cite{b-m} it will be enough to
prove that 
\begin{equation}
\lim_{t\rightarrow\infty}\mathcal{K}(u_{t})=\inf_{\mathcal{H}(X,\omega)}\mathcal{K}>-\infty\label{eq:flow minim seq}
\end{equation}
 To this end first we first recall that following inequality of Chen
\cite{chen}: 
\begin{equation}
\mathcal{K}(u)-\mathcal{K}(v)\leq d(u,v)Ca(u)^{1/2}\label{eq:k bounded by dist cal}
\end{equation}
 Moreover, as shown by Calabi-Chen (see \cite{ca-ch}) $d$ is decreasing
under the Calabi flow and hence 
\begin{equation}
d(u_{t},v_{t})\leq d(u_{0},v_{0})\label{eq:decreasing under cal flow}
\end{equation}
 for $u_{t}$ and $v_{t}$ evolving according to the Calabi flow \ref{eq:calabi flow on potentials}.
In particular, if we take $v_{0}:=u_{KE}$ as a potential of a Kähler-Einstein
metric $\omega_{KE},$ then $v_{t}=v_{0}$ and hence 
\begin{equation}
\mathcal{K}(u_{t})-\mathcal{K}(u_{KE})\leq d(u_{0},u_{KE})Ca(u_{t})^{1/2}\label{eq:mab smaller than ca}
\end{equation}
 Next, observe that there is a sequence $t_{j}$ such that 
\begin{equation}
Ca(u_{t_{j}})\rightarrow0\label{eq:pf of conv of cala}
\end{equation}
 as $t_{j}\rightarrow\infty.$ Indeed, by the variational formula
for $\mathcal{K}$ we have 
\begin{equation}
\frac{d\mathcal{K}(u_{t})}{dt}=-Ca(u_{t})\leq0\label{eq:mab along cal fl}
\end{equation}
 Hence, if it would be the case that $Ca(u_{t})\geq\epsilon>0$ as
$t\rightarrow\infty$ then this would force $\mathcal{K}(u_{t})\rightarrow-\infty$
as a $t\rightarrow\infty$ which contradicts the assumption that $\mathcal{K}(u)$
be proper and in particular bounded from below. This proves the claim
\ref{eq:pf of conv of cala} and hence, by \ref{eq:mab smaller than ca},
we also get 
\begin{equation}
\lim_{t_{j}\rightarrow\infty}\mathcal{K}(u_{t_{j}})\leq\mathcal{K}(u_{KE})=\inf_{\mathcal{H}(X,\omega)}\mathcal{K}\label{eq:subseq for mab along flow}
\end{equation}
 where the last property is a special case of Prop \ref{pro:minim of k-ene for pos theta}.
Finally, by \ref{eq:mab along cal fl} $\mathcal{K}(u_{t})$ is decreasing
and hence the previous inequality implies the inequality \ref{eq:flow minim seq},
finishing the proof of the theorem.
\begin{rem}
The previous proof gave the weak convergence of $\omega_{u_{t}},$
which is equivalent  to the $L^{1}-$convergence of the normalized
potentias $u_{t}-\sup u_{t}.$ But in fact the $L^{1}-$convergence
holds for $u_{t}$ (i.e. without normalising). Indeed, by the monotonicity
and properness of $\mathcal{K}$ we have that $J_{\omega}(u_{t})\leq C.$
Since, $d\mathcal{E}_{\omega}(u_{t})/dt=0$ this means that $\int u_{t}\omega^{n}\leq C'.$
But it follows from standard compactness arguments (for example used
in \cite{bbgz}) that $\{J_{\omega}\geq C\}\cap\{\int(\cdot)\omega^{n}\leq C'\}$
is relatively compact in $PSH(X,\omega)$ and hence so is the set
$\{u_{t}\},$ showing that there is no need to normalise $u_{t}.$ 
\end{rem}
One final remark about the twisted case:
\begin{rem}
\label{rem:twisted calabi flow}The previous proof admits a straight-forward
generalization to the setting of twisted Kähler-Einstein metrics when
$\theta\geq0,$ where $R_{\omega}$ is replaced by the trace of the
twisted Ricci curvature. Indeed, if $\theta\geq0$ the twisted functional
$\mathcal{K}_{\theta}$ is still geodesically convex (see Prop \ref{pro:convexity of twisted mab})
which at least formally implies \ref{eq:k bounded by dist cal} and
\ref{eq:decreasing under cal flow}. Hence the Hessian of $\mathcal{K}_{\theta}$
(defined wrt the metric $g$ above) is a semi-positive Hermitian operator
which implies that the corresponding flow decreases the length of
any initial curve and is hence distance decreasing (compare the proofs
in \cite{ca-ch} and \cite{ch3}). The estimate \ref{eq:mab smaller than ca}
is more involved as it requires a notion of \emph{weak} $C^{1,1}-$geodesics,
but the proof is a simple modification of the argument in \cite{chen}. 
\end{rem}

\section{\label{sub:K=00003D0000E4hler-Einstein-metrics-on log} Log Fano
manifolds and Donaldson's equation}

In this section we will consider the twisted Kähler-Einstein setting
when $\beta<0$ in the singular case when the twisting form $\theta$
is a linear combination of the integration currents along codimension
one analytic subvarieties in $X,$ i.e. 
\[
\theta:=\sum c_{i}\delta_{D_{i}},
\]
 where $D_{i}$ is an irreducible subvariety in $X.$ In other words,
$D_{i}$ is an irreducible effective divisor and we write 
\begin{equation}
\Delta:=\sum c_{i}D_{i},\label{eq:def of divisor delta}
\end{equation}
 for the corresponding $\R-$divisor on $X$ (abusing notation slightly
we will also denote its support by $\Delta).$ We will assume that
the the $D_{i}:s$ are distinct and smooth with simple normal crossings
(i.e. there are local coordinates where $D_{i}=\{z_{m(i)}=0\})$ and
$0<c_{i}<1.$ In the language of the minimal model program in algebraic
geometry this means that the \emph{log pair} $(X,\Delta)$ is\emph{
klt} (Kawamata Log Terminal). The measure $\mu_{0}$ in formula \ref{eq:mu zero in twisted setting}
is then well-defined and may be written as 
\begin{equation}
\mu_{0}=\mu_{\Delta}:=\prod_{i}|s_{i}|^{2c}dV\label{eq:measure def by klt div}
\end{equation}
 for some volume form $dV$ on $X,$ where $s_{i}$ is a section of
a holomorphic line bundle $L_{D_{i}}$ cutting out $D_{i}$ and $|\cdot|$
denote fixed smooth metrics on $L_{D_{i}}.$ The equation \ref{eq:cohom relation}
then translates to $[\omega]=c_{1}(-(K_{X}+L_{\Delta}))$ wich is
hence assumed to be a Kähler class (i.e. the pair $(X,\Delta)$ defines
a log Fano manifold). By Prop \ref{pro:reg} any finite energy solution
$u$ of the corresponding mean field equation is locally bounded.
Moreover, the current $\omega_{u}$ satisfies the following singular
Kähler-Einstein equation (to simplify the notation we set $\beta=-1):$
\begin{equation}
\mbox{Ric}\,\omega_{u}=\omega_{u}+\delta_{\Delta}\label{eq:singular k-e}
\end{equation}
 in the sense of currents (where $\mbox{Ric}\omega_{u}$ now denotes
the curvature current of the induced singular metric on $-K_{X}).$
We will mainly be concerned here with the case when $\Delta=(1-t)D,$
where $t>0$, $D$ is a smooth divisor. As is well-known, in the special
case when $t=1/m$ the pair $(X,\Delta)$ determines an orbifold structure
on $X$ with codimension one stabilizers $\Z/m\Z.$ Then \ref{eq:singular k-e}
in particular holds for any Kähler-Einstein metric on $X$ which is
smooth in the orbifold sense, which from a differential geometric
point of view means that $\omega$ has cone angles $2\pi1/m$ in the
directions transverse to $D$ (see for example the discussion in \cite{r-t}). 

In our general setting we define the\emph{ alpha-invariant of the
pair $(X,\Delta)$} by 
\[
\alpha(X,\Delta):=\alpha(-c_{1}(K_{X}+\Delta),\mu_{\Delta}).
\]
 In the orbifold case $\alpha(X,\Delta)$ coincides with the alpha-invariant
(i.e. the log canonical threshold) of the orbifold associated to $(X,\Delta)$
and was studied by Demailly-Kollar \cite{d-j}.

Applying Theorem \ref{thm:alpha invariant intro} combined with Kolodziej's
regularity theorem (just as in the proof of Theorem \ref{thm:existence etc when beta pos})
now gives the first statement in the following corollary concerning
global continuity. To obtain smoothness on $X-\Delta$ we will show
that the solution is the limit of smooth solutions to the twisted
Kähler-Einstein equations obtained by replacing the current $\Delta$
with a sequence of regularizations. 
\begin{cor}
\label{cor:alpha for pairs gives ec}Let $(X,\Delta)$ be a pair as
above and assume that 
\begin{equation}
\alpha(X,\Delta)>\frac{n}{n+1}\label{eq:alpha in statement cor alp pair}
\end{equation}
 Then there is a unique Hölder continuous solution $u$ to equation
\ref{eq:singular k-e}. \textup{\emph{Moreover, $\omega_{u}$ is a
smooth Kähler-Einstein metric on $X-\Delta$ and globally on $X$
it is a Kähler current, i.e. there is a Kähler form $\omega_{0}$
on $X$ such that $\omega_{u}\geq\omega_{0}$ on $X.$ When $(X,\Delta)$
defines an orbifold $\omega_{u}$ is smooth in the orbifold sense.}}\emph{ }\end{cor}
\begin{proof}
The existence of a Hölder continuous solution $u$ is a special case
of Theorem \ref{thm:alpha invariant intro} combined with Kolodziej's
result (just as in the proof of Prop \ref{pro:reg}). The uniqueness
follows from the very recent results in \cite{bern} (compare the
proof of Theorem \ref{thm:don eq} below). 

\emph{Higher order regularity when $\theta:=\delta_{\Delta}\geq0:$}

Let $\Theta\in c_{1}(L_{\Delta})$ be the curvature form of the fixed
smooth metric on the $\R-$line bundle $L_{\Delta},$ and let $u_{\Delta}^{(j)}:=\log(\sum|s_{i}|^{2c_{i}}+1/j).$
Then $\theta_{j}:=\Theta+dd^{c}u_{\Delta}^{(j)}$ is a sequence of
Kähler forms converging to $\delta_{\Delta},$ Take $u_{j}$ to be
a sequence of minimizers, normalized so that $\sup_{X}u_{j}=0,$ of
the corresponding twisted Mabuchi functionals $\mathcal{K}_{\theta_{j}}.$
Since $\alpha(-c_{1}(K_{X}+L_{\Delta}))\geq\alpha((-c_{1}(K_{X}+L_{\Delta}),\mu_{\Delta})(:=\alpha(X,\Delta))>n/(n+1)$
such minimizers exist and are smooth according to Thm \ref{cor:alpha inv for twisted}
and satisfy 
\begin{equation}
\frac{\omega_{u_{j}}^{n}}{n!V}=\frac{e^{-u_{j}}\mu_{\Delta}^{(j)}}{\int_{X}e^{-u_{j}}\mu_{\Delta}^{(j)}};\,\,\,\,\mbox{Ric \ensuremath{\omega_{u_{j}}=}}\omega_{u_{j}}+\theta_{j}\label{eq:pf of smooth sol for pair}
\end{equation}
 where $\mu_{\Delta}^{(j)}$ are volume forms on $X$ increasing to
the measure $\mu_{\Delta}.$ We may (after perhaps passing to a subsequence)
assume that $u_{j}\rightarrow u_{\infty}$ in $L^{1}(X).$

\emph{Step 1:} $J_{\omega}(u_{j})\leq C,\,\,\,\int e^{-(1+\epsilon)u_{j}}\mu_{\Delta}\leq C$

This is proved exactly as in the proof of Theorem \ref{cor:alpha inv for twisted}
using that $\mu_{\Delta}^{(j)}\leq\mu_{\Delta}$ and the assumed bound
on the alpha-invariant of $(-c_{1}(K_{X}+L_{\Delta}),\mu_{\Delta}).$

\emph{Step 2: The sequence $u_{j}$ is an asymptotic minimizer of
$K_{\theta}$ (and hence $\omega_{u_{j}}\rightarrow\omega_{u_{\infty}}$
solving equation \ref{eq:singular k-e})}

This also follows as before using that $\mu_{\Delta}^{(j)}\leq\mu_{\Delta}.$ 

\emph{Step 3: $\sup_{X}\left|u_{j}\right|\leq C$}

By the first equation in \ref{eq:pf of smooth sol for pair} and step
1 above we have that $\omega_{u_{j}}^{n}/\omega_{0}^{n}$ is uniformly
bounded in $L^{(1+\epsilon)}(X,\omega_{0}^{n})$ and hence Kolodziej's
theorem \cite{ko} gives the desired $C^{0}-$bound

\emph{Step 4: $(a)\,\omega_{u_{j}}\geq\frac{1}{C}\omega_{0}\,\mbox{\,\ on\,\ensuremath{X},\,\,\,}(b)\,\sup_{K}\left|\omega_{u_{j}}\right|_{\omega_{0}}\leq C_{K}$
$\mbox{on\,}$ $K\subset\subset X-\Delta$}

First observe that since $\theta_{j}\geq0$ equation \ref{eq:pf of smooth sol for pair}
shows that the Ricci curvature of $\omega_{u_{j}}$ is uniformly bounded
from below on $X$ (by a positive constant, but a negative constant
would also be fine for the following argument). Combined with the
uniform bound on $u_{j}$ in the previous step it follows from an
argument in \cite{b-k} which is a variant of the usual Aubin-Yau
Laplacian estimate \cite{au,y} that $(a)$ holds (the author learned
the argument from \cite{rub1} where it used to handle another situation
where $\mbox{Ric \ensuremath{\omega_{u_{j}}}}$is uniformly bounded
from below) . We next recall the argument: it follows directly from
the Chern-Lu (in)equality that 
\[
\Delta_{\omega_{u_{j}}}(\log(Tr_{\omega_{u_{j}}}\omega_{0})\geq-C(Tr_{\omega_{u_{j}}}\omega_{0})
\]
 using that there is a positive lower bound of the Ricci curvature
of $\omega_{u_{j}}$ and where $C$ is the upper bound of the bisectional
curvature of $\omega_{0}.$ Since, $\mbox{Ric \ensuremath{\omega_{u_{j}}\geq\omega_{u_{j}}}}$
it follows that there is a constant $C$ independent of $u_{j}$ such
that, setting $v_{j}:=Tr_{\omega_{u_{j}}}\omega_{0},$ we have
\begin{equation}
\Delta_{\omega_{u_{j}}}(\log v_{j}-(C+1)u_{j})\geq-(C+1)n+v_{j}\label{eq:second order yau ineq}
\end{equation}
 Evaluating the inequality above at a point where $\log v_{j}-(C+1)u_{j}$
attains its maximum (so that the lhs above is non-positive) and using
that $u_{j}$ is, by Step 3 above, uniformly bounded gives an upper
bound on $\sup_{X}v_{j}$ which implies the desired lower bound on
$\omega_{u_{j}}.$ Next, by equation \ref{eq:pf of smooth sol for pair}
and Step 3 above we have that $\omega_{u_{j}}^{n}/\omega_{0}^{n}$
is uniformly bounded from above on any fixed compact set $K$ in $X-\mbox{supp}\Delta$
which finishes the proof of Step 4.

Step 5: $\exists\alpha>0:\,\,\left\Vert u_{j}\right\Vert _{C^{2,\alpha}(K)}\leq C\mbox{\,\ on}K\subset\subset X-\Delta$

Given the previous estimates which, in particular, show that $\left\Vert u_{j}\right\Vert _{L^{\infty}(K)}\leq C,$
$\left\Vert \Delta_{\omega_{0}}u_{j}\right\Vert _{L^{\infty}(K)}\leq C$
and $MA(u_{j})\geq1/C,$ step 5 follows from a complex version of
the Evans-Krylov-Trudinger theory for local non-linear elliptic equations
(see Thm 5.1 in \cite{bl2}).

Finally, using the standard linear elliptic local (Schauder) estimates
and bootstrapping shows that $\left\Vert u_{j}\right\Vert _{C^{p,\alpha}(K)}\leq C_{p}$
for any $p>0$ and hence (after perhaps passing to a subsequence)
it follows that $u_{j}\rightarrow u_{\infty}$ in the $\mathcal{C}^{\infty}-$topology
on compacts on $X-\Delta.$ In particular, this shows that $u_{\infty}$
is smooth on $X-\Delta.$ 
\end{proof}
It may be worth pointing out that the variational part of the proof
above (i.e. Step 2) is not really needed as the rest of the argument
anyway produces a bounded function $u_{\infty}$ on $X$ satisfying
the limiting version of the Monge-Ampère equation \ref{eq:pf of smooth sol for pair}
on $X-\Delta$ and hence everywhere since the support of $\Delta$
is a pluripolar set. But one of the main virtues of the variational
approach is that it gives the convergence of \emph{any} sequence $u_{j}$
which is an asymptotic maximizer of the corresponding twisted Mabuchi
functional (under the usual properness assumption). In particular,
the previous corollary can be made more precise giving a singular
variant (i.e applied to $\theta=\delta_{\Delta})$ of the second point
of Thm \ref{cor:alpha inv for twisted} obtained by replacing the
volume form $dV$ used in the exponential integral of $u_{\infty}$
with the measure $\mu_{\Delta}.$

In the orbifold case Cor \ref{cor:alpha for pairs gives ec} is essentially
due to Demailly-Kollar who obtained a solution $\omega$ which is
a Kähler metric in the orbifold sense \cite{d-j}. Strictly speaking
the results in \cite{d-j} where formulated in the classical orbifold
setting of stabilizers of codimension $>1$ (then $X$ has quotient
singularities), but the same arguments are valid in the codimension
one case.

\subsection{\label{sec:Donaldson's-equation-and}Donaldson's equation and the
proof of Theorem \ref{thm:don eq} }

The existence of solutions to Donaldson's equation \ref{eq:donaldsons equ}
will be deduced from the criterion in Cor \ref{cor:alpha for pairs gives ec}
concerning the alpha-invariant of a pair $(X,\Delta)$ and the following
lower bound on such invariants in the particular setting of Donaldson's
equation. One of the ingredients in the proof is a an extension to
pairs of the well-known identification between alpha-invariants and
log canonical thresholds (see the appendix). 
\begin{prop}
\label{pro:Lower bound on alpha of pair}Let $L$ be an ample line
bundle over $X$ and $s$ a holomorphic section of $L$ such that
$D:=\{s=0\}$ is a smooth divisor. Then 
\begin{equation}
\alpha(L,\mu_{(1-\gamma)D})\geq\min\{\gamma,\alpha(L),\alpha((L_{|D})\}\label{eq:lower bound on alpha for pair in prop}
\end{equation}
\end{prop}
\begin{proof}
By Proposition \ref{pro:alpha inv as lct} in the appendix it will
be enough to prove that if $s_{m}\in H^{0}(mL)$ then $-t(\frac{1}{m}\log|s_{m}|^{2})$
is locally integrable wrt $\frac{1}{|s|^{2(1-\gamma)}}dV$ for any
fixed $t$ strictly smaller than the rhs in \ref{eq:lower bound on alpha for pair in prop}.
To this end we first recall that following inequality, which is an
immediate consequence of the Ohsawa-Takegoshi extension theorem (see
Thm 2.1 in \cite{d-j} and references therein): If $u\in PSH(\Omega)$
such that $u$ is not identically $-\infty$ on the smooth connected
complex submanifold $\{s=0\}\subset\Omega\subset\C^{n}$ then, for
$\delta>0,$ 
\begin{equation}
\int_{U}e^{-u}\frac{1}{|s|^{2(1-\delta)}}dV_{n}\leq C_{\delta}\int_{\{s=0\}}e^{-u}dV_{n-1}\label{eq:inversion of adj}
\end{equation}
 on some neighborhood $U\subset\Omega$ containing $\{s=0\}$ (depending
on $u).$ Now take $s_{m}\in H^{0}(X,mL)$ and decompose$s_{m}=s^{\otimes l}\otimes s'$
where $l\leq m$ and $s'\in H^{0}((m-l)L)$ does not vanish identically
on $D:=\{s=0\}$ unless $l=m.$ In the case when $m=l$ the integral
$I_{t}$ is clearly finite as long as $t<\gamma.$ Otherwise the bound
$l/m<1$ translates to 
\[
e^{-t\frac{1}{m}\log|s_{m}|^{2}}\frac{1}{|s|^{2(1-\gamma)}}=e^{-t(\frac{l}{m}\log|s|^{2})}e^{-t(\frac{m-l}{m})\frac{1}{m-l}\log|s'|^{2}}\frac{1}{|s|^{2(1-\gamma)}}\leq
\]
 
\[
\leq e^{-t\frac{1}{m-l}\log|s'|^{2}}\frac{1}{|s|^{2(1-\delta)}}
\]
 for any fixed $t\leq\gamma-\delta.$ Since, $\frac{1}{m-l}\log|s'|^{2}$
is a psh weight on $L$ the inequality \ref{eq:inversion of adj}
gives that the function $e^{-t\frac{1}{m}\log|s_{m}|^{2}}\frac{1}{|s|^{2(1-\delta)}}$
is locally integrable in a neighborhood of $\{s=0\}$ as long as $t\leq\inf\{\gamma,\alpha(L_{|D})\}-\delta.$
Moreover, on the complement of a neighborhood of $\{s=0\}\subset X$
the factor $\frac{1}{|s|^{2(1-\gamma)}}$ is bounded and hence $e^{-t\frac{1}{m}\log|s_{m}|^{2}}\frac{1}{|s|^{2(1-\delta)}}$
is locally integrable there as long as $t<\alpha(L).$ All in all,
this means that the integral $I_{t}(\frac{1}{m}\log|s_{m}|^{2})$
is finite if $t\leq\min\{\gamma,\alpha(L),\alpha((L_{|D})\}-\delta.$
\ref{eq:lower bound on alpha for pair in prop}. 
\end{proof}
Before continuing with the proof of Theorem \ref{thm:don eq} we make
two remarks. First we note that it follows immediately from Hölder's
inequality that 
\[
\alpha(L,\mu_{(1-\gamma)D})\geq\gamma\alpha(L)
\]
But the point with the previous proposition is that it will allow
us to deduce the existence of a solution to Donaldson's equation for
$\gamma$ sufficently small without assuming that the classical alpha-invariant
is sufficently big, i.e. without assuming that $\alpha(L)>n/(n+1).$
Secondly, the lower bound in the previous proposition should be compared
with the trivial upper bound $\alpha(L,\mu_{(1-\gamma)D})\leq\min\{\gamma,\alpha(L)\}$
(just take $\psi:=\log|s|^{2}).$ In the one dimensional case when
$\psi_{\Delta}$ is defined by a divisor $\Delta$ as a (formula \ref{eq:def of divisor delta})
with $c_{i}<1$ and $V:=\deg L=1$ a slight modification of the proof
above gives 
\begin{equation}
\alpha(L,\mu_{\Delta})=\min_{i}\{\alpha(L,1-c_{i}\}=\min_{i}\{1,1-c_{i}\}\label{eq:alpha inv troy}
\end{equation}
 (this also follows from the argument in example \ref{exa:fractal m-t}
since $\exp(-t(g_{x_{0}}))$ is integrable wrt $\mu_{\Delta}$ iff
$t<\min_{i}\{1,1-c_{i}\}).$

\subsubsection{The proof of Theorem \ref{thm:don eq}}

\emph{Existence:}

By a simple rescaled version of Corollary \ref{cor:alpha for pairs gives ec}
there is a solution if 
\[
\alpha(-(K_{X}),\mu_{(1-\gamma)D})>\gamma\frac{n}{n+1}
\]
 and by the previous Proposition \ref{pro:Lower bound on alpha of pair}
this inequality is clearly satisfied if $\gamma<\Gamma:=\frac{n+1}{n}\min\left\{ \alpha(-K_{X}),\alpha((-K_{X})_{|D})\right\} .$

\emph{Uniqueness: }

According to Berndtsson's very recent generalized Bando-Mabuchi uniqueness
theorem \cite{bern} there is a unique solution of Donaldson's equation
\ref{eq:donaldsons equ} unless there is a non-trivial holomorphic
vector field $V$ on $X$ which is tangent to $D$ (formally this
is a consequence of the strict convexity in Prop \ref{pro:convexity of twisted mab},
but the problem is the non-existence of bona fida geodesics connecting
two critical points). Next, we give a direct argument (which does
not rely on the previous existence result) showing that such a $V$
does not exist. Assume to get a contradiction that $V$ as above does
exist and take $\gamma$ sufficiently small (so that $0<\gamma<\Gamma).$
As shown above $\mathcal{K}_{(1-\gamma)D}$ is proper wrt energy (since
the condition on the alpha-invariant of $(X,(1-\gamma)D)$ is satisfied).
Hence it will, to reach a contradiction, be enough to find a curve
$u_{t}$ such that $J_{\omega}(u_{t})$ tends to infinity, but $\mathcal{K}_{(1-\gamma)D}(u_{t})$
does not. To this end we let $u_{t}$ be defined by $u_{t}:=-\log(h_{t}/h)$
where $h_{0}$ is a fixed metric on $-K_{X}$ with curvature form
equal to the Kähler metric $\omega$ and $h_{t}:=F_{t}^{*}h_{0}$
where $F_{t}$ denotes the lift to $-K_{X}$ of the flow defined by
$V.$ Then $u_{t}$ satisfies the geodesic equation \ref{eq:geod equ},
where $V_{t}$ coincides with $V,$ the given holomorphic vector field
(compare \cite{d00}). Setting $J(t):=J_{\omega_{0}}(u_{t})$ a direct
calculation gives 
\[
\frac{d^{2}J(t)}{d^{2}t}=\int_{X}\partial_{t}^{2}u_{t}\frac{\omega^{n}}{n!}=\int_{X-D}|V_{t}|_{\omega_{t}}^{2}\frac{\omega^{n}}{n!}>0
\]
 if $V$ is non-trivial and hence $J(t)\rightarrow\infty$ as $|t|\rightarrow\infty.$
Finally, Prop \ref{pro:convexity of twisted mab} implies that $\mathcal{K}_{(1-\gamma)D}(u_{t})$
is affine wrt $t.$ Hence, the limit of $\mathcal{K}_{(1-\gamma)D}(u_{t})$
is bounded from above when either $t\rightarrow\infty$ or $t\rightarrow-\infty$
giving the desired contradiction.

\emph{Regularity of the curve $\gamma\mapsto\omega_{\gamma}$}

Fix $\gamma=\gamma_{0}\in]0,\Gamma].$ Since the (normalized) potential
$u_{\gamma}$ of the Kähler-Einstein current $\omega_{\gamma}$ maximizes
the functional $\mathcal{G}_{\gamma}(:=\mathcal{G}_{-\gamma,(1-\gamma)D})$
it is not hard to check that $\mathcal{G}_{\gamma_{0}}(u_{t})$ converges,
when $\gamma\rightarrow\gamma_{0},$ to the supremum of $\mathcal{G}_{\gamma_{0}}$
(this is similar to the proof of step 2 in the proof of Theorem \ref{thm:don eq})
and hence it follows, just like in the Step 2 in the proof of Cor
\ref{cor:alpha for pairs gives ec}, that any limit point in the $L^{1}-$closure
of $\{u_{\gamma}\}$ is a maximizer of $\mathcal{G}_{t_{0}}.$ By
the uniqueness in the previous point this means that $\omega_{\gamma}\rightarrow\omega_{\gamma_{0}}$
in the sense of currents. Finally, to prove the stronger continuity
it is enough to show that, for any positive integer $m,$ the partial
derivatives of $u_{\gamma}$ total order $m$ are uniformly bounded
on a given compact subset $K$ in $X-D$ with a constant which is
independent of $\gamma.$ But this follows from writing $u_{\gamma}$
as the limit of $u_{\gamma}^{(j)}(=:u_{j})$ where $u_{j}$ was defined
in the proof of the previous corollary (where higher order estimates
were obtained with constants which are clearly independent of $\gamma).$ 
\begin{rem}
In the case of a Riemann surface Cor \ref{cor:alpha for pairs gives ec}
combined with the simple identity \ref{eq:alpha inv troy} gives a
new proof of Troyanov's existence result for metrics with constant
positive curvature and conical singularities (\cite{tr}, Thm C).
Note that the proof in \cite{tr} was also variational, but our derivation
of the corresponding Moser-Trudinger inequality is new (the proof
in \cite{tr} uses weighted Sobolev inequalities).
\end{rem}
After the first version of the present paper appeared on ArXiv there
have been several important new developments concerning Kähler-Einstein
metrics with conical singularities along a divisor that we next briefly
describe, referring to the cited papers for precise statements. In
the paper \cite{do-3} Donaldson established the openness property
with respect to the strictly positive parameter $\gamma$ of solutions
to equation \ref{eq:donaldsons equ} with certain further regularity
properties (defined using weighted Hölder spaces adapted to $D)$.
Using Donaldson's result and a perturbation trick in \cite{j,ma}
Brendle \cite{br} proved the existence of Ricci flat metrics with
conical singularities along a given divisor, assuming that $\gamma\in]0,1/2].$
\footnote{Combining the arguments in \cite{do-3,br} with those in the present
paper the author then noted that $\omega_{\gamma}$ has conical singularities
for any $\gamma$ sufficiently small thus confirming Donaldson's conjecture
(see arXiv:1011.3976 {[}v3{]}). More precisely the result was shown
to hold for $\gamma$ < $\min\{\Gamma,1/2\})$ by deforming any orbifold
solution. Here we have omitted the argument as the subsequent results
\cite{jmr} permit to remove the unnatural restriction $\gamma<1/2$
(as explained above). %
} A very general existence and regularity theory for Kähler-Einstein
with conical singularities along a divisor (or in other words Kähler-Einstein
e\emph{dge metrics}) has been developed by Jeffres-Mazzeo-Rubinstein
\cite{jmr} based on the edge calculus combined with a continuity
method. In particular, in the positively curved case, the results
in \cite{jmr} say that if the twisted Mabuchi functional corresponding
to a pair $(X,(1-\gamma)D)$ is proper then there is a Kähler-Einstein
metric with appropriate cone singularities and a complete asymptotic
expansion along $D,$ only assuming that $\gamma\in]0,1]$ (we refer
to \cite{jmr} for the precise regularity statement and the corresponding
function spaces). Since we have shown that the properness does hold
for $\gamma<\Gamma$ in Donaldson's setting, i.e. when $D$ is an
anti-canonical divisor, the results in \cite{jmr} hence imply that
the solutions in $\omega_{\gamma}$ in Theorem \ref{thm:don eq} indeed
always have conical singularities and moreover admit a complete asymptotic
expansion. In another direction Campana-Guenancia-P\u{a}un \cite{cgh}
used a direct regularization argument to produce negatively curved
Kähler-Einstein metrics with cone singularities along a given klt
divisor $\Delta=\sum_{i}(1-\gamma_{i})D_{i},$ assuming $\gamma_{i}\in]0,1/2[.$

\section{Appendix: Alpha-invariants and log canonical thresholds for pairs}

In this appendix we will extend the results of Demailly in \cite{dem2}
concerning alpha-invariants of Kähler classes to a more singular setting
and in particular to the setting of klt pairs considered in section
\ref{sub:K=00003D0000E4hler-Einstein-metrics-on log}. The main point
is the simple observation that only very basic $L^{2}-$estimates,
as compared to \cite{dem2}, are needed for our purposes. 

We will assume that $[\omega]=c_{1}(L)$ for an ample line bundle
$L$ and we fix a smooth Hermitian metric $|\cdot|$ on $L$ with
positive curvature form $\omega.$ As usual, we equip $PSH(X,\omega)$
with its $L^{1}-$topology. Let $\mu$ be a finite measure on $X$
such that 

\begin{equation}
\mu(=\mu_{v}):=e^{-v}dV\label{eq:m delta in append}
\end{equation}
where $v$ is a quasi-psh function on $X,$ i.e. $v\in PSH(X,\epsilon\omega)$
for some $\epsilon>0.$ For a fixed positive number $t$ we consider
the functional 
\[
I_{t,v}(\psi):=\int e^{-t\psi}\mu_{v}
\]
 on the space $PSH(X,\omega).$ By definition 
\[
\alpha(L,\mu_{v}):=\sup\left\{ t:\, I_{t,v}\,\,\mbox{is\,\,\ bounded\,\ from\,\ above\,\ on \ensuremath{PSH(X,\omega)\cap\{\sup_{X}(\cdot)=0\}}}\right\} 
\]
Let us start be recalling the following fundamental local result from
\cite{d-j} which will allow us to replace the uniform boundedness
of $I_{t}$ with finiteness.
\begin{thm}
\label{thm:(Demailly-Kollar-).-Let}(Demailly-Kollar ). Let $K$ be
a compact subset in a domain $\Omega\subset\C^{n}$ and $u\in PSH(\Omega).$
Define $c_{K}(u)$ as the sup over all $c\geq0$ such that $e^{-cu}$
is integrable on some neighborhood of $K.$ If $u_{j}\rightarrow u$
in $L^{1}(\Omega)$ where $u_{j}\in PSH(\Omega),$ then $e^{-cu_{j}}\rightarrow e^{-cu}$
in $L^{1}$ on some neighborhood of $K$ for any $c$ such that $c<c_{K}(u).$ 
\end{thm}
Applying this theorem to the present global setting gives the following 
\begin{cor}
\label{cor:If-the-functional}\label{cor:global cont of exp int}If
the functional $I_{(t+\epsilon),(1+\delta)v}$ is finite on $PSH(X,\omega)$
for some $\epsilon,\delta>0$ then $I_{t,v}$ is continuous. As a
consequence, 
\begin{itemize}
\item Given $\delta>0$ and $t<\alpha(L,e^{-(1+\delta)v})$ the functional
$I_{t,v}$ is continuous on $PSH(X,\omega)$ (wrt the $L^{1}-$topology). 
\item $\alpha(L,\mu_{v}):=\sup\left\{ t:\, I_{t,v}<\infty\,\mbox{on \ensuremath{PSH(X,\omega)}}\right\} $ 
\end{itemize}
\end{cor}
\begin{proof}
Take $t$ and $\epsilon,\delta>0$ such that $I_{t+\epsilon,(1+\delta)v}$
is finite on $PSH(X,\omega).$ Assume that $\psi_{j}\rightarrow\psi$
in $PSH(X,\omega)$ and normalize so that $\sup_{X}\psi=0.$ For any
fixed point $x$ with a small neighborhood $U$ we may apply the previous
theorem to $u_{j}:=\psi+v/t+C|z|^{2}$ for $C$ sufficently large
and deduce that $e^{-t\psi_{j}}e^{-v}\rightarrow e^{-t\psi}e^{-v}$
in $L^{1}(U).$ Using a partition of unity hence shows that $I_{t,v}$
is continuous on $PHS(X,\omega).$ This immediately implies the first
point in the corollary. To prove the second point we let $\alpha^{*}(L,\mu_{v})$
be defined as the rhs in the second point. Clearly, $\alpha^{*}(L,\mu_{v})\geq\alpha(L,\mu_{v})$
and by the first point and the compactness of the space $PSH(X,\omega)\cap\{\sup_{X}(\cdot)=0)$
we also have $\alpha(L,\mu_{v})\geq\alpha^{*}(L,\mu_{(1+\delta)v})$
for any $\delta>0.$ The proof is now concluded letting $\delta$
tend to zero and noting that the rhs above is continuous in $\delta,$
which follows from the fact that $\alpha^{*}(L,\mu_{\lambda v})$
is concave in $\lambda$ (I am greatful to Sebastien Boucksom for
pointing this out to me). Indeed, by Hölder's inequality, the function
$f_{\psi}(t,\lambda)=\log I_{t,\lambda v}(\psi)$ is convex in $(t,\lambda)$
and hence $\alpha_{\psi}^{*}(\lambda):=\sup\{t:\, I_{t,\lambda v}(\psi)<\infty\}$
is concave in $\lambda.$ Taking the infimum over all $\psi$ thus
shows that $\alpha^{*}(L,\mu_{\lambda v})$ is concave in $\lambda,$
as desired. \end{proof}
\begin{lem}
The functional $I_{t}$ above is finite on $PHS(X,\omega)$ iff it
is finite on the subspace of all singular weights of the form $\psi=\frac{1}{m}\log|s_{m}|^{2}$
for $s_{m}\in H^{0}(X,mL),$ where $m$ is positive integer. \end{lem}
\begin{proof}
The ``only if'' direction is trivial and hence we fix $\psi\in PSH(X,\omega).$
By replacing $\psi$ with $(1-\delta)\psi+\delta\psi_{0}$ it is enough
to prove that $I_{t}$ is finite on the space of all $\psi$ such
that $\omega_{\psi}\geq\delta\omega$ for some $\delta>0.$ The proof
of the lemma is based on the observation that one may replace the
volume form $dV_{\omega}$ used in the proof of $(iii)$ in Theorem
A.4 in \cite{dem2} with any measure $\mu$ which the following property:
for any weight $\psi$ as above 

\begin{equation}
\left\Vert s\right\Vert _{(m\psi,\mu)}^{2}:=\int_{X}|s|^{2}e^{-m\psi}\mu,\label{eq:hilbert space}
\end{equation}
 defines a Hilbert norm on the $N_{m}-$dimensional subspace \emph{$\mathcal{H}_{m}:=\left\Vert \cdot\right\Vert _{(m\psi,\mu)}^{2}<\infty$}
of $H^{0}(X,mL),$ with $N_{m}>0$ for $m$ sufficently large. To
see that this is the case for $\mu$ satisfying \ref{eq:m delta in append}
we rewrite $\left\Vert s\right\Vert _{(m\psi,\mu)}^{2}=\left\Vert s\right\Vert _{(\tilde{\psi}_{m},dV)}^{2},$
where $\tilde{\psi}_{m}=:m\psi+v.$ Since $v$ is quasi-psh we have
we that $\tilde{\psi}_{m}\in PSH(X,m\omega\epsilon/2)$ for $m$ sufficenty
large. This means that $|\cdot|^{2}e^{-\psi_{m}}$ defines a singular
Hermitian metric on $mL$ with a curvature current bounded form below
by $m\omega\epsilon/2.$ But then it follows from well-known $L^{2}-$estimates
for $\bar{\partial}$ (see \cite{dem2} and references therein for
much more precise results) that for any $m$ sufficenty large there
exists $s\in H^{0}(X,mL)$ for $m$ such that $\left\Vert s_{m}\right\Vert _{(m\psi,dV)}^{2}<\infty.$
We can now proceed exactly as in the proof of $(iii)$ in Theorem
A.4 in \cite{dem2}. Indeed, let $\psi_{m}\in PSH(X,\omega)$ be defined
by 
\[
\psi_{m}:=\frac{1}{m}\sup_{s_{m}\in H^{0}(X,mL)}\log\frac{|s_{m}|^{2}}{\left\Vert s_{m}\right\Vert _{(m\psi,\mu)}^{2}}=\frac{1}{m}\log\sum_{i=1}^{N_{m}}|s_{m}^{(i)}|^{2}
\]
 where $s_{m}^{(i)}$ is an orthonormal base for $\mathcal{H}_{m}$
and set $\alpha_{m}:=\sup\{t:\, I_{t}(\psi_{m})<\infty\}.$ Then 
\begin{equation}
1/\alpha(L,\mu)\leq1/\alpha_{m}+1/m\label{eq:proof of alpha is log}
\end{equation}
 To see this one writes $e^{-\frac{m}{p}\psi}=e^{\frac{m\psi_{m}-m\psi}{p}}e^{\frac{m\psi_{m}}{p}}$
for a fixed $p>1$ and apply Hölder's inequality with dual exponents
$(p,q)$ giving 
\[
\int e^{-\frac{m}{p}\psi}\mu\leq(\int e^{m\psi_{m}}\mu)^{1/p}(\int e^{-\frac{mq}{p}\psi_{m}})^{1/q}
\]
 By the second equality in the definition of $\psi_{m}$ above the
first factor is a constant $(=N_{m}^{1/p})$ and the second factor
is finite as long as $\frac{mq}{p}<\alpha_{m},$ i.e $(\frac{m}{p})^{-1}<\frac{1}{\alpha_{m}}+\frac{1}{m}.$
Since $p>1$ was arbitary this proves \ref{eq:proof of alpha is log}.

Now take $t$ such that $I_{t}$ is finite for all $\psi$ of the
form $\frac{1}{m}\log|s_{m}|^{2}.$ By the second equality in the
definition of $\psi_{m}$ above combined with the concavity of log
we the deduce that $I_{t}(\psi_{m})$ is finite for any $m$ sufficently
large and hence $\alpha(\psi_{m})\geq t.$ All in all this means that
$\alpha(L,\mu)\geq t(1+\epsilon_{m}),$ where $\epsilon_{m}\rightarrow0$
and hence letting $m\rightarrow\infty$ finishes the proof of the
proposition. 
\end{proof}
All in all we arrive att the following 
\begin{prop}
\label{pro:alpha inv as lct}\textup{Let $\mu$ be a measure satisfying
}\ref{eq:m delta in append}\textup{. Then the invariant $\alpha(L,\mu)$
coincides with the sup over all positive numbers $t$ such that $\int_{X}e^{-t\frac{1}{m}\log|s_{m}|^{2}}\mu$
is finite for all $s_{m}\in H^{0}(X,mL)$ and $m\in\N.$ In particular,
this is the case for the measure $\mu:=\mu_{\Delta}$ associated to
a klt divisor $\Delta$ (formula \ref{eq:measure def by klt div}). }
\end{prop}
Formulated in terms of log canonical thresholds (see \cite{dem2})
the previous proposition amounts to the identity 

\[
\alpha(L,\mu_{\Delta})=\inf_{D_{m}}\mbox{lct}_{X}(X,D_{m}+\Delta),
\]
 where $m$ is a positive integer and $D_{m}$ is the zero divisor
of some $s_{m}\in H^{0}(X,mL).$ 
\begin{rem}
All the previous results apply in the more general case when $L$
is big, i.e. $\omega$ is only assumed to be a Kähler current (just
as in \cite{dem2}). The proofs are essentially the same.\end{rem}


\begin{thebibliography}{10}
\bibitem{au}Aubin, T.: Equations du type Monge-Amp`ere sur les vari\textasciiacute{}et\textasciiacute{}es
k\textasciidieresis{}ahl\textasciiacute{}eriennes compactes, Bull.
Sci. Math. 102 :1 (1978) 63-95

\bibitem{b-m}Bando, S; Mabuchi, T: Uniqueness of Einstein Kahler
metrics modulo connected group actions, Algebraic geometry, Sendai,
1985, Adv. Stud. Pure Math., vol. 10, North-Holland, Amsterdam, 1987,
pp. 11-40.

\bibitem{b-k}Bando, S; Kobayashi, R: Ricci-flat Kähler metrics on
affine algebraic manifolds II. Math. Ann. 287 (1990), pp. 175\textendash{}180

\bibitem{Bec}Beckner, W: Sharp Sobolev inequalities on the sphere
and the Moser-Trudinger inequality. Annals of Math. 138 (1993), 213-242.

\bibitem{b0}Berman, R.J. Analytic torsion, vortices and positive
Ricci curvature. arXiv:1006.2988

\bibitem{b1}Berman, R.J.:Kähler-Einstein metrics emerging from free
fermions and statistical mechanics. J. of High Energy Phys. (JHEP),
Vol. 2011, Issue 10 (2011) arXiv:1009.2942

\bibitem{b-b}Berman, R; Boucksom, S; Growth of balls of holomorphic
sections and energy at equilibrium. Invent. Math. 181 (2010), no.
2, 337-394

\bibitem{bbgz}Berman, R.J.: Boucksom, S; Guedj, V; Zeriahi, A: A
variational approach to complex Monge-Ampère equations. Publications
mathématiques de l'IHÉS (to appear). arXiv:0907.4490

\bibitem{b-b-w}Berman, R.J:; Boucksom, S; Witt Nyström, D: Fekete
points and equidistribution on complex manifolds. Acta Math. (to appear).
Preprint at arXiv:0907.2820

\bibitem{b-d}Berman, R.J:; Demailly, J-P: Regularity of plurisubharmonic
upper envelopes in big cohomology classes. in ``Perspectives in Analysis,
Geometry, and Topology'',\emph{ }Springer-Verlag. arXiv:0905.1246

\bibitem{bern}Berndtsson, B: A Brunn-Minkowski type inequality for
Fano manifolds and the Bando-Mabuchi uniqueness theorem. arXiv:1103.0923

\bibitem{begz}Boucksom, S; Essidieux,P: Guedj,V; Zeriahi: Monge-Ampere
equations in big cohomology classes. Acta. Math. (to appear). arXiv:0812.3674

\bibitem{bl2}B\l{}ocki, Z: The Calabi-Yau theorem, to appear in Lecture
Notes in Mathematics as a part of the volume Complex Monge-Ampère
equations and geodesics in the space of Kähler metrics (ed. V. Guedj).
http://gamma.im.uj.edu.pl/\textasciitilde{}blocki/publ/

\bibitem{bl-k}B\l{}ocki, Z; Ko\l{}odziej, S: On regularization of
plurisubharmonic functions on manifolds, Proceedings of the American
Mathematical Society 135 (2007), 2089-2093.

\bibitem{br}Brendle, S: Ricci flat Kahler metrics with edge singularities.
arXiv:1103.5454

\bibitem{ca}Calabi, E: Extremal Kähler metrics, in Seminar on Differential
Geometry, volume 102 of Ann. of Math. Stud., pages 259\textendash{}290,
Princeton Univ. Press, Princeton, N.J., 1982.

\bibitem{ca-ch}Calabi, E; Chen, X.X.: The space of K\textasciidieresis{}ahler
metrics, 2, J. Differential Geom., 61(2):173\textendash{}193, 2002.

\bibitem{cgh}Campana, F; Guenancia, H; P\u{a}un, M: Metrics with
cone singularities along normal crossing divisors and holomorphic
tensor fields. arXiv:1104.4879

\bibitem{cao-1}Cao, H.D: Deformation of Kähler metrics to Kähler-Einstein
metrics on compact Kähler manifolds. Invent. Math. 81 (1985), no.
2, 359--372.

\bibitem{cao}Cao, Huai-Dong Existence of gradient Kähler-Ricci solitons.
Elliptic and parabolic methods in geometry (Minneapolis, MN, 1994),
1\textendash{}16,

\bibitem{c-l}Carlen, E and Loss, M: Competing symmetries, the logarithmic
HLS inequality and Onofri's inequality on Sn. Geometric and Functional
Analysis 2 (1992) 90\textendash{}104.

\bibitem{ch2}Chen, X.X: On the lower bound of the Mabuchi energy
and its application. Internat. Math. Res. Notices 2000, no. 12, 607--623.

\bibitem{ch3}Chen, X-X.: A New Parabolic Flow in Kähler Manifolds.
Comm. Anal. Geom. 12 (2004), no. 4,

\bibitem{chen}Chen, X-X. Space of K\textasciidieresis{}ahler metrics
III\textendash{} The greatest lower bound of the Calabi energy, Invent.
math. 175, 453-680(2009).

\bibitem{c-h}Chen, X.X; He, W.Y: On the Calabi Flow. Amer. J. Math.
Vol. 130, Num. 2, l 2008

\bibitem{c-t-z}Chen, X. X.; Tian, G; Zhang, Z: On the weak Kähler-Ricci
flow. arXiv:0802.0809.

\bibitem{chr}Chrusciel, P.T.: Semi-global existence and convergence
of solutions of the Robinson- Trautman (2-dimensional Calabi) equation,
Comm. Math. Phys. 137 (1991), no. 2, 289\textendash{}313.

\bibitem{ci}Cianchi, A: Moser-Trudinger trace inequalities. Adv.
in Math. Vol. 217, Issue 5, 20 2008, p. 2005--2044

\bibitem{de-ze}Dembo, A; Zeitouni, O: Large deviations techniques
and applications. Jones and Bartlett Publishers, Boston, MA, 1993.
xiv+346 pp.

\bibitem{dem0}Demailly, J. P.: Regularization of closed positive
currents and intersection theory. J. Alg. Geom. 1 (1992), no. 3, 361\textendash{}409.

\bibitem{d-j}Demailly, J-P; Kollar, J: Semi-continuity of complex
singularity exponents and Kähler-Einstein metrics on Fano orbifolds.
Ann. Sci. École Norm. Sup. (4) 34 (2001), no. 4, 525\textendash{}556.

\bibitem{dem2}Demailly, J-P: Appendix to I. Cheltsov and C. Shramov's
article ``Log canonical thresholds of smooth Fano threefolds\textquotedblright{}
: On Tian's pr and log canonical thresholds. Uspekhi Mat. Nauk, 63:5(383)
(2008), 73\textendash{}180

\bibitem{dem1}Demailly, J.-P.: Complex analytic and algebraic geometry;
manuscript Institut Fourier, first edition 1991, available online
at http://www-fourier.ujf-grenoble.fr/\textasciitilde{}demailly/books.html
.

\bibitem{d-s}Dinh, T-CV; Nguyên, V-A; Sibony, N: Exponential estimates
for plurisubharmonic functions. J. Differential Geom. Volume 84, Number
3 (2010), 465-488.

\bibitem{din}Ding, W.: Remarks on the existence problem of positive
Kähler-Einstein metrics. Math. Ann. 463--472 (1988).

\bibitem{d-t}Ding, W. and Tian, G.: The generalized Moser-Trudinger
Inequality. Proceedings of Nankai International Conference on Nonlinear
Analysis, 1993.

\bibitem{d00}Donaldson, S.K: Symmetric spaces, Kähler geometry and
Hamiltonian dynamics. Northern California Symplectic Geometry Seminar,
13--33, Amer. Math. Soc. Transl. Ser. 2, 196, Amer. Math. Soc., Providence,
RI, 1999.

\bibitem{do}Donaldson, S.K.: Conjectures in Kähler geometry, in Strings
and geometry, 71\textendash{}78, Clay Math. Proc. 3, AMS 2004.

\bibitem{do-2}Donaldson, S.K.: \textquotedbl{}Discussion of the
Kahler-Einstein problem\textquotedbl{}. Available at http://www2.imperial.ac.uk/\textasciitilde{}skdona/KENOTES.PDF

\bibitem{do-3}Donaldson, S.K.: Kahler metrics with cone singularities
along a divisor. arXiv:1102.1196

\bibitem{egz}Eyssidieux, E; Guedj, E: Zeriahi, A: Singular Kähler-Einstein
metrics. arXiv:math/0603431. J. Amer. Math. Soc. 22 (2009), 607-639.

\bibitem{fine}Fine, J., Constant scalar curvature K\textasciidieresis{}ahler
metrics on fibred complex surfaces, J. Differential Geom. 68 (2004),
no. 3, 397-432.

\bibitem{cher}Fontana, L.: Sharp borderline Sobolev inequalities
on compact Riemannian manifolds, Comment. Math. Helv. 68 (1993), 415-454.

\bibitem{f-m}Fontana, L; Morpurgo, C: Adams inequalities on measure
spaces. arXiv:0906.5103

\bibitem{g-z}Guedj,V; Zeriahi, A: Intrinsic capacities on compact
Kähler manifolds. J. Geom. Anal. 15 (2005), no. 4, 607--639.

\bibitem{g-z2}Guedj,V; Zeriahi, A: The weigthed Monge-Ampère energy
of quasiplurisubharmonic functions. J. Funct. An. 250 (2007), 442-482.

\bibitem{hi}Hiep, P.H.: Holder continuity of solutions to the complex
Monge-Ampere equations on compact Kahler manifolds. . arXiv:0904.4145.
To appear in Annales de l'Institut Fourier

\bibitem{j}Jeffres, T: Uniqueness of K\textasciidieresis{}ahler-Einstein
cone metrics, Publ. Mat. 44, 437\textendash{}448 (2000)

\bibitem{j2}Jeffres, T: Schwarz Lemma for Kähler Cone Metrics. Int
Math Res Notices (2000).

\bibitem{jmr}Jeffres, T; Mazzeo, R; Rubinstein, Y.A:; Kähler-Einstein
metrics with edge singularities. Arxiv 1105.5216

\bibitem{ko}Ko\l{}odziej, S.: The complex Monge\textendash{}Ampère
equation. Acta Math. 180, 69\textendash{}117 (1998).

\bibitem{ko2}Ko\l{}odziej, S.: Hölder continuity of solutions to
the complex Monge\textendash{}Ampère equation with the right-hand
side in L p : the case of compact Kähler manifolds

\bibitem{li}Li, H: On the lower bound of the K-energy and F-functional.
Osaka J. Math. Volume 45, Number 1 (2008), 253-264.

\bibitem{m2}Mabuchi, T: K-energy maps integrating Futaki invariants.
Tohoku Math. J. (2) 38 (1986), no. 4, 575--593.

\bibitem{m1}Mabuchi, T: Some symplectic geometry on compact Kähler
manifolds. I, Osaka Journal of Mathematics 24 (1987), 227\textendash{}252.

\bibitem{ma}Mazzeo, R: Kähler-Einstein metrics singular along a smooth
divisor. Journées \textquotedbl{}Équations aux Dérivées Partielles''
(Saint-Jean-de-Monts, 1999), Exp. No. VI, 10 pp., Univ. Nantes, Nantes,
1999.

\bibitem{na}Nadel, A.M.: Multiplier ideal sheaves and Kähler-Einstein
metrics of positive scalar curvature, Annals of Mathematics 132 (1990),
549-596.

\bibitem{p-s+}Phong, D.H: Song, J; Sturm, J; Weinkove, B: The Moser-Trudinger
inequality on Kahler-Einstein manifolds. Amer. J. Math. 130 (2008),
no. 4, 1067-1085, arXiv:math/0604076

\bibitem{rub0}Rubinstein, Y.A: On energy functionals, Kahler-Einstein
metrics, and the Moser-Trudinger-Onofri neighborhood, J. Funct. Anal.
255, special issue dedicated to Paul Malliavin (2008), 2641-2660.

\bibitem{ru}Rubinstein, Y.A: On the construction of Nadel multiplier
ideal sheaves and the limiting behavior of the Ricci flow. Trans.
Amer. Math. Soc. 361 (2009), 5839-5850.

\bibitem{rub1}Rubinstein, Y.A: Some discretizations of geometric
evolution equations and the Ricci iteration on the space of Kahler
metrics Adv. Math. 218 (2008), 1526-1565.

\bibitem{s-t 0}Song, J; Tian, G: The Kähler-Ricci flow on surfaces
of positive Kodaira dimension. Inv. Math. 170 (2007), 609-653.

\bibitem{s-t}Song, J; Tian, G: Canonical measures and Kahler-Ricci
flow. arXiv:0802.2570

\bibitem{s-t-2}Song, J; Tian, G: The Kahler-Ricci flow through singularities.
arXiv:0909.4898

\bibitem{st}Stoppa, J: Twisted constant scalar curvature Kähler metrics
and Kähler slope stability. J. Differential Geom. Volume 83, Number
3 (2009), 663-691.

\bibitem{sz-to}Székelyhidi, G; Tosatti.V: Regularity of weak solutions
of a complex Monge-Ampère equation. To appear in Analysis \& PDE

\bibitem{r-t}Ross, J; Thomas, R. P.: Weighted projective embeddings,
stability of orbifolds and constant scalar curvature Kähler metrics.
arXiv:0907.5214

\bibitem{ta}Tarantello, G: Selfdual gauge field vortices. An analytical
approach. Progress in Nonlinear Differential Equations and their Applications,
72. Birkhäuser Boston, Inc., Boston, MA, 2008. xiv+325 pp.

\bibitem{ti1}Tian, G: On Kähler-Einstein metrics on certain K\textasciidieresis{}ahler
manifolds with C1(M) > 0, Inventiones Mathematicae 89 (1987), 225\textendash{}246.

\bibitem{t-y}Tian, G; On Calabi's conjecture for complex surfaces
with positive first Chern class. Invent. Math. Vol. 101, Nr. 1 (1990)

\bibitem{ti2}Tian, G. K\textasciidieresis{}ahler-Einstein metrics
with positive scalar curvature, Invent. Math. 130 (1997), no. 1, 1\textendash{}37.

\bibitem{ti}Tian, G: Canonical Metrics in Kähler Geometry, Birkh\textasciidieresis{}auser,
2000.

\bibitem{ti3}Tian, G: Kahler-Einstein metrics on algebraic manifolds,
in: Transcendental methods in algebraic geometry (Cetraro 1994), 143--185.

\bibitem{t-y0}Tian, G.; Yau, S.-T. Existence of Kähler-Einstein metrics
on complete Kähler manifolds and their applications to algebraic geometry.
Mathematical aspects of string theory (San Diego, Calif., 1986), 574\textendash{}628,
Adv. Ser. Math. Phys., 1, World Sci. Publishing, Singapore, 1987. 

\bibitem{t-y-1}Tian, G; Yau, S-T: Complete Kähler manifolds with
zero Ricci curvature, I. J. Amer. Math. Soc. 3 (1990), no. 3, 579\textendash{}609.

\bibitem{t-z}Tian, G; Zhu, X: Convergence of Kähler-Ricci flow. J.
Amer. Math. Soc. 20 (2007), no. 3

\bibitem{tr2}Troyanov, M: Metrics of constant curvature on a sphere
with two conical singularities. Differential geometry, Proc. 3rd International
Symposium on Differential Geom. (Peniscola, 1988), Lecture Notes in
Math, Vol 1410, Springer-Verlag, 296-308

\bibitem{tr}Troyanov, M: Prescribing curvature on compact surfaces
with conical singularities, Trans. AMS, 324 (1991) 793-821

\bibitem{y}Yau, S.T: On the Ricci curvature of a compact Kähler manifold
and the complex Monge-Ampère equation. I. Comm. Pure Appl. Math. 31
(1978), no. 3, 339--41

\bibitem{ze}Zeriahi, A: Volume and capacity of sublevel sets of a
Lelong class of psh functions. Indiana Univ. Math. J. 50 (2001), no.
1, 671\textendash{}703. \end{thebibliography}
\end{document}